\documentclass[10pt,reqno]{amsart}

\usepackage{amssymb}
\usepackage{fullpage}
\usepackage{graphicx}
\usepackage{hyperref}

\usepackage{setspace}
\theoremstyle{plain}
\newtheorem{thm}{Theorem}[section]
\newtheorem{cor}[thm]{Corollary}
\newtheorem{lem}[thm]{Lemma}

\numberwithin{equation}{section}

\newcommand{\T}{\mathbb{T}}
\newcommand{\G}{\mathcal{G}_{n,d}}
\newcommand{\F}{\mathcal{F}}
\newcommand{\E}[1]{\mathbb{E}\left[#1\right]}
\newcommand{\pr}[1]{\mathbb{P}\left[#1\right]}
\newcommand{\Est}[1]{\mathbb{E}^*\left[#1\right]}
\newcommand{\ind}[1]{\mathbf{1}_{\{ #1 \}}}
\newcommand{\eps}{\epsilon}
\newcommand{\Lm}{\Lambda}
\newcommand{\al}{\alpha_d}
\newcommand{\ro}{\rho_d}
\newcommand{\rt}{\circ}
\newcommand{\prc}{\mathbf{Y}}
\newcommand{\iid}{\mathbf{X}}
\newcommand{\cplbl}{\mathbf{W}}
\newcommand{\ber}{\mathrm{Ber}_d}
\newcommand{\den}[1]{\mathrm{den}(#1)}
\newcommand{\corr}[1]{\mathrm{corr}(#1)}
\newcommand{\avdeg}[1]{\mathrm{avdeg}(#1)}

\title{Factor of IID percolation on trees}

\author{Mustazee Rahman}

\address[Mustazee Rahman]
{Department of Mathematics\\
University of Toronto\\
40 St.~George Street\\
Toronto\\
ON M5S 2E4\\
Canada}
\email[Mustazee Rahman]{mustazee@math.toronto.edu}

\keywords{factor of iid process, induced forest, percolation, random regular graphs, local algorithms.}

\begin{document}

\begin{abstract}
We study invariant percolation processes on the $d$-regular tree that are obtained as a factor
of an iid process. We show that the density of any factor of iid site percolation process
with finite clusters is asymptotically at most $(\log d)/d$ as $d \to \infty$. This bound is asymptotically
optimal as it can be realized by independent sets. One implication of the result is a $(1/2)$-factor
approximation gap, asymptotically in $d$, for estimating the density of maximal induced forests in
locally tree-like $d$-regular graphs via factor of iid processes.
\end{abstract}

\maketitle

\section{Introduction} \label{sec:intro}
Let $\T_d$ denote the rooted $d$-regular tree where a distinguished vertex $\rt$ is the root.
Let $\chi$ be a finite set of `colours'. A \emph{factor of iid} (FIID) process on $\T_d$ is an
invariant random process $\Phi$ taking values in $\chi^{V(\T_d)}$ that is defined as follows.
Let $f : [0,1]^{V(\T_d)} \to \chi$ be a measurable function with respect to the product Borel $\sigma$-algebra
of $[0,1]^{V(\T_d)}$. Suppose $f$ is also invariant under root-preserving automorphisms, i.e.,
$f(\gamma \cdot \omega) = f (\omega)$ for every $\omega \in [0,1]^{V(\T_d)}$ and any
$\gamma \in \mathrm{Aut}(\T_d)$ satisfying $\gamma(\rt) = \rt$. \big(The automorphisms act
of $[0,1]^{V(\T_d)}$ by $\gamma \cdot \omega(v) = \omega(\gamma^{-1}v)$\big).
A \emph{random labelling} $\iid = (X(v), v \in V(\T_d))$ of $\T_d$ is a process consisting of independent, identically
distributed random variables $X(v)$ such that each $X(v) \sim \mathrm{Uniform}[0,1]$.
Since $\mathrm{Aut}(\T_d)$ acts transitively on $V(\T_d)$, for each vertex $v \in V(\T_d)$
let $\gamma_{v \to \rt}$ be an automorphism such that $\gamma_{v \to \rt}v = \rt$.
The process $\Phi = (\Phi(v), v \in V(\T_d))$ is then defined by
$$\Phi(v) = f(\gamma_{v \to \rt} \cdot \iid)\,.$$
The value of $f$ does not depend on the choice of $\gamma_{v \to \rt}$
because $f$ is invariant under root-preserving automorphisms.
The function $f$ is called the \emph{factor} associated to $\Phi$.
An FIID process $\prc$ is called a \emph{percolation process}
if $\prc \in \{0,1\}^{V(\T_d)}$, i.e., $\prc$ is $\{0,1\}$--valued at each vertex.

Informally speaking, the factor is a rule that decides the value at the root by considering
the randomly labelled rooted tree as the input. Given any other vertex $v$
the value of the process at $v$ is determined by shifting $v$ to the root,
which naturally permutes the labels, and then applying the rule for the root.
The prototypical example of FIID percolation is Bernoulli site percolation:
for each $p \in [0,1]$ the factor associated to Bernoulli site percolation on $\T_d$
of density $p$ is $f(\omega) = \ind{\omega(\rt) \leq p}$.

FIID processes over $\T_d$, or more generally over Cayley graphs of finitely generated non-amenable groups,
have been of interest in probability theory, combinatorics, computer science and ergodic theory.
Recently, they have been used as randomized algorithms to construct and estimate important
graph parameters such as maximal independent sets, perfect matchings, large cut sets, eigenvectors, etc.
See \cite{CGHV, CL, EL, GG, HV, HW, LW, Lyons, LN} and the references therein for
examples. We explain in Section \ref{sec:proj} how FIID processes on $\T_d$ serve as a model of
\emph{local algorithms} for random $d$-regular graphs (or for any sequence of locally tree-like $d$-regular graphs).

FIID processes also arise naturally in the theory of sparse graph limits as developed by
Hatami, Lov\'{a}sz and Szegedy \cite{HLS}. The authors pose the question of what processes
over large graphs can be realized as FIID processes on $\T_d$, and make the conjecture
that all reasonable optimization problems on random $d$-regular graphs can be approximated by
FIID processes on $\T_d$ as the size of the graphs tend to infinity \cite[Conjecture 7.13]{HLS}.
The conjecture was refuted for maximum independent sets by Gamarnik and Sudan \cite{GS}.
They showed that maximal independent sets in random $d$-regular graphs cannot be approximated
by FIID independent sets in $\T_d$ for large $d$. More specifically, they proved that the maximal
density of FIID independent sets in $\T_d$ is smaller than that of random $d$-regular graphs by a
multiplicative factor of less than 1 irrespective of the graphs' size, provided that
$d$ is sufficiently large. An asymptotically optimal gap by a factor of $1/2$, as $d \to \infty$, was established
in \cite{RV}. On the other hand, the conjecture holds for perfect matchings \cite{CL, LN}
and covariance structures \cite{BSV}.

In this paper we study FIID percolation processes and show as a consequence of the main theorem
that maximal induced subgraphs with bounded size components on random
$d$-regular graphs cannot be realized as FIID percolation processes on $\T_d$.
Due to the local tree-like nature of random $d$-regular graphs, induced subgraphs with
bounded components will be induced forests with high probability. On the other hand, any
induced subgraph of $\T_d$ is a forest. Our main result (Theorem \ref{thm:main}) then implies that
FIID induced forests in $\T_d$ whose components are finite with probability one cannot
model the maximal induced forests in random $d$-regular graphs if $d$ is large. This is
interesting because the best known lower bounds to the maximum size density of induced forests
in random regular graphs is derived by way of FIID percolation processes on regular trees having
finite components \cite{HWlwbdd}. Since independent sets are induced forests having components
of size one, our result generalizes the aforementioned results about independent sets.
In order to state our results precisely we introduce some terminology.

The \emph{density} of an FIID percolation process $\prc$ on $\T_d$ is
$$\den{\prc} = \pr{\prc (\rt) = 1} = \E{f(\iid)}\,.$$
The density does not depend on the choice of the root due to invariance of $\prc$.
The \emph{average degree} of $\prc$ is the expected degree of the root,
after being conditioned to be in $\prc$:
$$ \avdeg{\prc} = \E{\# \{ v \in N_{\T_d}(\rt) : \prc(v) = 1\}\mid \prc(\rt) = 1}.$$
The main result of this paper states that if $\avdeg{\prc} = o(\log d)$
as $d \to \infty$ then $\den{\prc} \leq (1+ o(1))\frac{\log d}{d}$. A
corollary is that when the components of the subgraph induced by $\{ v \in V(\T_d): \prc(v) = 1\}$
are finite with probability one then the density of $\prc$ is asymptotically at most $(1 + o(1)) \frac{\log d}{d}$ as $d \to \infty$.
This is because if the components are finite then the average degree of $\prc$ is at most 2.
Our bound is optimal as $d \to \infty$ because there exists FIID independent sets in $\T_d$
that have density asymptotically equal to $(\log d)/d$ as $d \to \infty$ \cite{GG, LW, Shearer}.

An important concept in the proof of the aforementioned result is a measure of correlation
between the value of the root and one of its neighbours in a FIID percolation process $\prc$. 
Let $\rt'$ be a fixed neighbour of the root $\rt$. The \emph{correlation ratio} of $\prc$ is defined as
\[ \corr{\prc} = \frac{\pr{\prc(\rt) = 1, \prc(\rt') = 1}}{\pr{\prc(\rt) = 1} \cdot \pr{\prc(\rt') = 1}}
=  \frac{\pr{\prc(\rt) = 1, \prc(\rt') = 1}}{\den{\prc}^2}\,.\]
Note that the usual correlation between $\prc(\rt)$ and $\prc(\rt')$ is
$\frac{\pr{\prc(\rt) = 1, \prc(\rt') = 1}-\den{\prc}^2}{\den{\prc}\cdot(1-\den{\prc})}$.
In this paper we deal instead with the correlation ratio as defined above because it arises
naturally in the course of the analysis. The correlation ratio is related to the average
degree as $\avdeg{\prc_d} = d \, \den{\prc} \corr{\prc}$. Our main result is the following.

\begin{thm} \label{thm:main}
Let $\{ \prc_d \}$ be a sequence of FIID percolation processes where $\prc_d$ is a
percolation process on $\T_d$. If $\corr{\prc_d} \to 0$ as $d \to \infty$ then
\[ \limsup_{d \to \infty} \, \frac{\den{\prc_d}}{(\log d)/d} \leq 1\,.\]
Alternatively, if $\avdeg{\prc_d} = o(\log d)$ as $d \to \infty$ then
$ \limsup_{d \to \infty} \, \frac{\den{\prc_d}}{(\log d)/d} \leq 1$.
\end{thm}
The result stated in terms of the average degree follows from the one stated in terms of the correlation ratio
because the the aforementioned formula that connects to these two quantities. Assuming,
without loss of generality, that $\den{\prc_d} \geq (\log d)/d$ we deduce that if
$\avdeg{\prc_d} = o(\log d)$ then $\corr{\prc_d} \to 0$.

\begin{cor} \label{cor:fincomp}
Let $\{ \prc_d \}$ be a sequence of FIID percolation processes with $\prc_d$ defined on $\T_d$.
Suppose that the connected components of the subgraph of $\T_d$ 
induced by the vertices $\{v \in V(\T_d): \prc_d(v) = 1\}$ are finite with probability one. Then
\[\limsup_{d \to \infty} \, \frac{\den{\prc_d}}{(\log d)/d} \leq 1\,.\]
\end{cor}
Let us now explain the relevance of this result to maximum induced forests in random $d$-regular graphs
as it relates to the conjecture of Hatami, Lov\'{a}sz and Szegedy.
For an integer $\tau \geq 1$ define $\alpha^{\tau}(G)$ to be the maximum value of $|S|/|G|$
where $S \subset V(G)$ has the property that the induced subgraph $G[S]$ has components of
size at most $\tau$. Let $\G$ denote a random $d$-regular graph on $n$ vertices. Bayati, Gamarnik
and Tetali \cite{BGT} show that $\E{\alpha^{1}(\G)}$ converges as $n \to \infty$.
Employing their argument it can be deduced that $\E{\alpha^{\tau}(\G)}$ converges
to a limit $\alpha^{\tau}(d)$ for every $\tau$.
It is natural to consider the quantity $\alpha^{\mathrm{finite}}(d) = \sup_{\tau} \alpha^{\tau}(d)$ and compare it to the quantity
$$\alpha_{\mathrm{FIID}}^{\mathrm{finite}}(d) = \sup \{ \den{\prc}: \prc \;\text{is a percolation process on}\;
\T_d \;\text{with finite components}\}.$$

Due to the locally tree-like nature of random regular graphs the quantity $\alpha^{\mathrm{finite}}(d)$
measures the size density of the largest induced forests in $\G$ as $n \to \infty$. The quantity
$\alpha_{\mathrm{FIID}}^{\mathrm{finite}}(d)$ measures the largest size density of induced forests
in $\G$ that are generated from FIID processes (explained further in Section \ref{sec:proj}).
The conjecture of Hatami, Lov\'{a}sz and Szegedy would imply that
$\alpha^{\mathrm{finite}}(d) = \alpha_{\mathrm{FIID}}^{\mathrm{finite}}(d)$ for every $d$.
Corollary \ref{cor:fincomp} and the existence of FIID independent sets in $\T_d$ of density
$(1 + o(1))\frac{\log d}{d}$ implies
$$\alpha_{\mathrm{FIID}}^{\mathrm{finite}}(d) = (1+o(1)) \frac{\log d}{d} \;\;\text{as}\;\; d \to \infty.$$

It is known that $\alpha^{\mathrm{finite}}(d) = (2 + o(1)) \frac{\log d}{d}$ as $d \to \infty$ \cite{Rahman}.
It has been actually been known much earlier that the independence ration $\alpha^1(d) = (2+o(1)) \frac{\log d}{d}$
as $d \to \infty$ \cite{Bol, FL}. From these results we conclude that
$$\lim_{d \to \infty} \; \dfrac{\alpha_{\mathrm{FIID}}^{\mathrm{finite}}(d)}{\alpha^{\mathrm{finite}}(d)} = \frac{1}{2}.$$
This shows a $(1/2)$-factor approximation gap as $d \to \infty$ for estimating the maximum density of induced
forests in random $d$-regular graphs by using FIID induced forest in $\T_d$ with finite components.

Two natural questions arise following Theorem \ref{thm:main}. First, is there a bound on the density
of FIID percolation processes on $\T_d$ in terms of the correlation ratio? Suppose $\{ \prc_d\}$
is a sequence of FIID percolation such that $\corr{\prc_d} \to c$. What can be concluded about
$\den{\prc_d}$ in terms of $c$? Second, what is the maximal density of an FIID percolation on the
\emph{edges} of $\T_d$ if it has finite clusters? The answer to the second question is due to
H\"{a}ggstr\"{o}m \cite{Hag} who proved that the maximal density is $2/d$.

To address the first question, note that there is no bound on the density if $c = 1$. Indeed, Bernoulli percolation on $\T_d$
has correlation ratio $1$ and density $p$ for any desired $p \in [0,1]$. Also, Lyons \cite{Lyons} constructs
FIID percolation processes on $\T_d$ having density arbitrarily close to $1/2$ and correlation ratio $1 - O(1/\sqrt{d})$ for large $d$.
This construction can be utilized to get FIID percolation processes on $\T_d$ having density of any fixed value
but having correlation ratio of order $1 - O(1/\sqrt{d})$ . On the other hand, if $c < 1$ then there is a bound on the
density of order $O((\log d)/d)$.

\begin{thm} \label{thm:corr}
Let $\{ \prc_d\}$ be a sequence of FIID percolation processes where $\prc_d$ is a percolation process on $\T_d$.
Let $\Psi(c) = 1 - c + c\log c$ for $c \geq 0$ ($\Psi(0) = 1$). If $\corr{\prc_d} \leq c < 1$ for every $d$, or
$\corr{\prc_d} \geq c > 1$ for every $d$ then for any $\eps > 0$ there exists $d_0 = d_0(c,\eps)$ such that
for $d \geq d_0$,
\[ \den{\prc_d} \leq \dfrac{2}{\Psi(c)d} \cdot \left ( \log d - \log \log d + 1 +\log \Psi(c) + \eps \right )\,.\]

\noindent Also, given $0 \leq c < 1$ there exists a sequence of percolations $\{ \mathbf{Z}_d \}$ with
$\corr{\mathbf{Z}_d} = c$ and
\[ \den{\mathbf{Z}_d} = \frac{1 - o(1)}{\sqrt{1-c}} \cdot \frac{\log d}{d} \quad \text{as}\; d \to \infty.\]
\end{thm}
We are not actually aware of a sequence of FIID percolation processes whose correlation ratios
satisfy $\Psi(\prc_d) \geq c > 1$ for every $d$. We also expect the bounds in Theorem \ref{thm:corr}
to be suboptimal but do not know how to improve upon them in a general setting.

\paragraph{\textbf{Outline of the paper}}
The proofs of Theorems \ref{thm:main} and \ref{thm:corr} are based on the connection between
FIID processes on $\T_d$ and their projections onto random $d$-regular graphs. We will
explain how to project an FIID percolation on $\T_d$ to a random $d$-regular
graph. This will result in a percolation process on random $d$-regular graphs whose local
statistics, such as the density and correlation ratio, are close to that of the original with high probability.
Then we will use combinatorial arguments to bound the  probability of observing a percolation process
on random $d$-regular graphs whose statistics are close to that of the original.
This leads to a fundamental non-negativity condition for an entropy-type functional 
that can be associated to any FIID percolation process. These results are developed in Section \ref{sec:rrg},
where the main result is the derivation of the aforementioned entropy inequality in Theorem \ref{thm:FIIDentropy}.
From there we proceed to show that the entropy functional is non-negative only if the density is small
with respect to the correlation ratio. This is done in Section \ref{sec:corrprf} to prove Theorem \ref{thm:corr}
and then in Section \ref{sec:mainprf} to prove Theorem \ref{thm:main}.

The basic strategy of the proof of Theorem \ref{thm:main} is similar to what has been done
for independent sets in \cite{RV}. Our techniques improve that result to the more general
question of percolation with finite clusters. Finally, we note that the entropy functional that we come
upon has been studied by Bowen \cite{Bwn}, who developed an isomorphism invariant for free group
actions on measure spaces in order to resolve questions in ergodic theory. Backhausz and Szegedy \cite{BS}
have also used similar entropy arguments to show that certain processes in random $d$-regular graphs
cannot be represented on trees.

\section{Entropy inequality for FIID processes} \label{sec:rrg}

In this section we prove an inequality that is satisfied by the entropy
of the local statistics of any FIID process $\Phi$ taking values in $\chi^{V(\T_d)}$.
The key result of this section is presented in Theorem \ref{thm:FIIDentropy}.
In order to prove this theorem we first explain how to project a FIID process
$\Phi \in \chi^{V(\T_d)}$ onto random $d$-regular graphs to obtain a random
$\chi$-valued colouring of random $d$-regular graphs such that the local statistics
of the coloring are close to the local statistics of $\Phi$ with high probability (Section \ref{sec:proj}).
This is actually what it means to model FIID processes on random regular graphs.
Following this we use combinatorial arguments to count colourings of random $d$-regular
graphs whose statistics agree with the statistics of the projection (Section \ref{sec:count}).
From this counting argument we can prove Theorem \ref{thm:FIIDentropy}.

Throughout this section the values of the degree $d$ and the set $\chi$ are fixed.
We consider asymptotic analysis in the parameter $n \to \infty$, where $n$ is the size
of the graphs to which we project $\Phi$.

We use the configuration model \cite[Chapter 2.4]{Bolbook} as the
probabilistic method to sample a random $d$-regular graph, $\G$, on $n$ labelled vertices.
It is sampled in the following manner. Each of the $n$ distinct vertices emit
$d$ distinct half-edges, and we pair up these $nd$ half-edges uniformly at random. (We
assume that $nd$ is even.) These $nd/2$ pairs of half-edges can be glued into full edges
to yield a random $d$-regular graph. Note that the resulting graph can have loops and
multiple edges. There are $(nd-1)!! = (nd-1)(nd-3)\cdots 3\cdot 1$ such graphs.
We denote this set of multigraphs by $G_{n,d}$; so $\G$ is a uniform random element
of $G_{n,d}$.

In order to project FIID processes on $\T_d$ onto $\G$ it is necessary to understand
the \emph{local structure} of $\G$. It is well known (see \cite[Chapter 9.2]{JLbook}) that
the number of cycles of length $\ell$ in $\G$ converges in moments to a Poisson random
variable with mean $(d-1)^{\ell}/2\ell$. Consequently, for every constant $L$, the expected
number of cycles in $\G$ of length at most $L$ remains bounded as $n \to \infty$.
It follows from this that $\G$ is locally tree-like in the sense that the fixed size
neighbourhoods of all but an order $O(1)$ number of vertices are tress with high probability.
Indeed, if the $r$-neighbourhood, $N_G(v,r)$, of a vertex $v$ in a graph $G$ is not a tree
then it contains a cycle of length at most $2r$. Thus, $N_G(v,r)$ is a tree if $v$ is not within
graph distance $r$ of any cycle in $G$ having length at most $2r$. Since the expected number
of cycles in $\G$ of length at most $2r$ is bounded in $n$, and the $r$-neighbourhood
of any vertex in $\G$ contains at most $d^r$ vertices, it follows that the expected number of
vertices $v$ in $\G$ such that $N_{\G}(v,r)$ is not a tree remains bounded in $n$ for every fixed $r$.
Therefore, if $\rt_n$ is a uniform random vertex of $\G$ then for every $r \geq 0$
\begin{equation} \label{eqn:locallimit}
\pr{N_{\G}(\rt_n,r) \cong N_{\T_d}(\rt,r)} = 1 - O(1/n) \quad \text{as}\; n \to \infty\,.
\end{equation}
This is what it means for $\G$ to be locally tree-like.

\subsection{Projecting FIID processes onto random regular graphs} \label{sec:proj}

Let $\T_{d,r}$ denote the rooted subtree $N_{\T_d}(\rt,r)$.
We say a factor $f$ is a \emph{block factor} if there exists a finite $r$ such that
$f$ maps $[0,1]^{V(\T_{d,r})}$ into $\chi$, i.e., $f$ does not depend of the labels
outside of $\T_{d,r}$. The smallest such $r$ is the \emph{radius} of the factor.
It is well known that given an FIID process $\Phi$ on $\T_d$ with factor $f$ it is possible to find
block factors $f_j$ such that the corresponding processes $\Phi^j$
converge to $\Phi$ in the weak topology \cite{Lyons}. As such, for any FIID percolation process $\prc$
there exist approximating percolation processes $\prc^j$ having block factors such
that $\den{\prc^j} \to \den{\prc}$ and $\corr{\prc^j} \to \corr{\prc}$.
Henceforth, we assume that all factors associated to FIID percolation processes on $\T_d$
are block factors.

We now explain how to project an FIID process $\Phi$ with block factor $f$ of radius $r$
onto a $d$-regular graph $G \in G_{n,d}$. The factor $f$ maps to a finite set $\chi$,
which we take to be $\{0,1,\ldots,q\}$. We begin with a random labelling
$\iid = (X(v), v \in V(G))$ of $G$. Given any vertex $v \in G$ if its $r$-neighbourhood
is a tree then set $\Phi_G(v) = f(X(u), u \in N_G(v,r))$. This is allowed since
$N_G(v,r) = \T_{d,r}$ by assumption. Otherwise, set $\Phi_G(v) = 0$.
The process $\Phi_G$ is the \emph{projection} of $\Phi$ onto $G$.
The projection $\Phi_G$ is also called a \emph{local algorithm} on $G$;
for example, if $\Phi$ is an independent set in $\T_d$ with probability one then
$\Phi_G$ would produce random independent sets in $G$ and be called
a local algorithm for independent sets in $G$.

If we project $\Phi$ to $\G$ then the projection $\Phi_{\G}$ has the property
that its local statistics converge to the local statistics of $\Phi$ as $n \to \infty$.
More precisely, let $(\rt_n,\rt'_n)$ be a uniform random directed edge of $\G$.
Let $\rt$ be the root of $\T_d$ with let $\rt'$ be a fixed neighbour. Then for any
$i,j \in \chi$, the probabilities
\begin{equation} \label{eqn:localstats}
\pr{\Phi_{\G}(\rt_n) = i, \Phi_{\G}(\rt'_n) = j} \to \pr{\Phi(\rt) = i, \Phi(\rt') = j} \quad \text{as}\; n \to \infty\,.
\end{equation}
Consequently, $\pr{\Phi_{\G}(\rt_n) = i} \to \pr{\Phi(\rt) = i}$ as well.
Note that $\rt_n$ is a uniform random vertex of $\G$.

To see (\ref{eqn:localstats}), observe that the value of the pair $(\Phi_{\G}(\rt_n), \Phi_{\G}(\rt'_n))$ depends only
on the labels on the $r$-neighbourhood of the edge $(\rt_n,\rt'_n)$. If this is a tree then
the $r$-neighbourhood of the edge $(\rt_n,\rt'_n)$ is isomorphic to the $r$-neighbourhood
of $(\rt,\rt')$. Therefore, $(\Phi_{\G}(\rt_n), \Phi_{\G}(\rt'_n))$ agrees with $(\Phi(\rt),\Phi(\rt'))$
if the labels on $N_{\G}((\rt_n,\rt'_n),r)$ are lifted to $N_{\T_d}((\rt,\rt'),r)$ in the natural way.
On the other hand, if $N_{\G}((\rt_n,\rt'_n),r)$ is not a tree
then $N_{\G}(\rt_n, r+1)$ is also not a tree since $N_{\G}((\rt_n,\rt'_n),r) \subset N_{\G}(\rt_n,r+1)$.
As $\G$ is locally tree-like as in (\ref{eqn:locallimit}),  we have that
$\pr{N_{\G}(\rt_n,r+1) \not\cong \T_{d,r+1}} = O(1/n)$.
As a result, it is straightforward to conclude that
$\big |\pr{\Phi_{\G}(\rt_n) = i, \Phi_{\G}(\rt'_n) = j} - \pr{\Phi(\rt) = i, \Phi(\rt') = j} \big| = O(1/n)$.

In the proof of Theorem \ref{thm:FIIDentropy} we need more than the convergence of local statistics.
We require that the random quantities $\# \{(u,v) \in E(\G): \Phi_{\G}(u) = i, \Phi_{\G}(v) = j\}$
are close to their expectations. Here, $(u,v)$ refers to a directed edge so $(u,v) \neq (v,u)$
unless $u=v$. Note that
\begin{equation*}
\pr{\Phi_{\G}(\rt_n) = i, \Phi_{\G}(\rt'_n) = j} = \E{\frac{\#\{(u,v) \in E(\G): \Phi_{\G}(u) = i, \Phi_{\G}(v) = j\}}{nd}} \,.
\end{equation*}
In other words, we require that the empirical values of the local statistics of
$\Phi_{\G}$ are close to their expected value with high probability.
To show this we use the Hoeffding--Azuma concentration inequality \cite[Theorem 1.20]{Bolbook}.
Consider a product probability space $(\Omega^n,\mu^n)$ and a bounded function $h : \Omega^n \to \mathbb{R}$
that is $L$-Lipschitz with respect to the Hamming distance in $\Omega^n$. The Hoeffding-Azuma inequality states that
\[ \mu^n\left ( \mid h - \mathbb{E}_{\mu^n}[h] \mid > \lambda \right ) \leq 2e^{-\frac{\lambda^2}{2L^2n}}. \]

To apply this to our case, fix a graph $G \in G_{n,d}$ and consider the probability space generated by
a random labelling $\iid$ of $G$. The function $h$ is taken to be $\# \{(u,v) \in E(G): \Phi_{G}(u) = i, \Phi_{G}(v) = j\}$.
We now verify that $h$ is Lipschitz as a function of the labels with Lipschitz constant at most $2d^{r+1}$.
If the value $X(w)$ of the label at $w$ is changed to $X'(w)$ then the value of the pair $(\Phi_{G}(u), \Phi_{G}(v))$ can
only change for edges $(u,v)$ that are within graph distance $r$ of $w$. Otherwise, the labels used to evaluate
$(\Phi_{G}(u), \Phi_{G}(v))$ does not contain $X'(w)$. Therefore, the number of directed edges $(u,v)$ where the
value of $\Phi_{G}$ can change is at most $2d^{r+1}$, which is a trivial upper bound on the number of directed edges that
meet $N_{G}(w,r)$. This implies that our function is $(2d^{r+1})$-Lipschitz.

From the Hoeffding--Azuma inequality we conclude that
\begin{align*}
\mathbb{P} \; \bigg [ \; & \bigg | \frac{\#\{(u,v) \in E(\G):
 \Phi_{\G}(u) = i, \, \Phi_{\G}(v) = j\}}{nd} \\
   \, - & \, \pr{\Phi_{\G}(\rt_n) = i, \Phi_{\G}(\rt'_n) = j} \bigg | > \eps \bigg ] \\
 \leq& \; 2 e^{- \frac{\eps^2 n}{8d^{2r}}}.
\end{align*}

Due to convergence of the expected value of local statistics as shown in (\ref{eqn:localstats})
we can replace the quantity $\pr{\Phi_{\G}(\rt_n) = i, \Phi_{\G}(\rt'_n) = j}$ in the inequality
above with the quantity $\pr{\Phi(\rt) = i, \Phi(\rt') = j}$ at the expense of adding a term
$\mathrm{err}_n$ such that $\mathrm{err}_n \to 0$ as $n \to \infty$. Then, after taking a
union bound over all $i,j \in \chi$ we deduce from the inequality above that for some
constant $C_{d,r,q}$ and error function $\mathrm{err}_n \to 0$ as $n \to \infty$, we have
\begin{align}
\mathbb{P} \bigg [ \max_{i,j \in \chi} \bigg | \frac{\#\{(u,v) \in E(\G):
\Phi_{\G}(u) = i, \Phi_{\G}(v) = j\}}{nd} -& \pr{\Phi(\rt) = i, \Phi(\rt') = j} \bigg | > \eps \bigg ] \nonumber \\
\leq& \, e^{- \frac{\eps^2 n}{C_{d,r,q}}} + \mathrm{err}_n. \label{eqn:concentration}
\end{align}

\subsection{Counting colourings of random regular graphs} \label{sec:count}
From (\ref{eqn:concentration}) we see that we can prove bounds on the local statistics
of $\Phi$ if we can use the structural properties of $\G$ to deduce bounds on the
empirical value of the local statistics of $\Phi_{\G}$. We achieve the latter by
combinatorial arguments.

On any graph $G \in G_{n,d}$ the process $\Phi_{G}$ creates
an ordered partition $\Pi$ of $V(G)$ with $q + 1 = \# \chi$ cells, namely, the sets
$\Pi(i) = \Phi_{G}^{-1}(\{i\})$ for $i \in \chi$. (It is ordered in the sense that the cells
are distinguishable.) Associated to $\Pi$ is its \emph{edge profile},
defined by the quantities

\begin{eqnarray*}
P(i,j) &=& \frac{\# \{ (u,v) \in E(G): \, \Phi_{G}(u) = i, \Phi_{G}(v) = j\}}{nd} \\
\pi(i) &=& \frac{\# \{ v \in V(G): \Phi_{G}(v) = i\}}{n}\,.
\end{eqnarray*}

As we have noted earlier, if we pick a uniform random directed edge $(\rt_n,\rt'_n)$ of $G$
then $P(i,j) = \pr{\Phi_{G}(\rt_n) = i, \Phi_{G}(\rt'_n)=j}$ and $\pi(i) = \pr{\Phi_{G}(\rt_n) = i}$.
In particular, the matrix $P = \big [P(i,j) \big ]_{\{i,j \in \chi\}}$ and the vector $\pi = (\pi(i))_{\{i \in \chi\}}$
are probability distributions over $\chi^2$ and $\chi$, respectively. Also, the marginal distribution of
$P$ along its rows is $\pi$, and $P$ is a symmetric matrix.

Our chief combinatorial ingredient, in particular, where we use the structural property of $\G$
is in computing the expected number of ordered partitions of $\G$ that admit a given edge profile.
From (\ref{eqn:concentration}) we see that any FIID process $\Phi$ admits, with high probability,
an ordered partition of $V(\G)$ with edge profile close to $P_{\Phi} = \big [\pr{\Phi(\rt) = i, \Phi(\rt') = j} \big]_{\{i,j\}}$
and $\pi_{\Phi} = (\pr{\Phi(\rt) = i})_{\{i\}}$. Thus, if the probability of observing a partition of $\G$ with this
edge profile is vanishingly small then the process $\Phi$ cannot exist on $\T_d$.
Bounding the probability of observing partitions with a given edge profile is difficult. Instead,
we bound the expected number of such partitions, which serves as an upper bound to the probability.

The \emph{entropy} of a discrete probability distribution $\mu$ is
$$H(\mu) = \sum_{x \in \mathrm{support}(\mu)} - \mu(x) \log \mu(x).$$

Let $Z(P,\pi) = Z(P,\pi, G)$ be the number of ordered partitions of $V(G)$ with edge profile $(P,\pi)$. 
\begin{lem} \label{lem:edgeprofile}
For the random graph $\G$,
$$\E{Z(P,\pi,\G)} \leq \mathrm{poly}(n) \times e^{n \big [\frac{d}{2}H(P) - (d-1)H(\pi) \big ]},$$
where $\mathrm{poly}(n)$ is a polynomial factor whose degree and coefficients depend only on $d$ and $q$. 
\end{lem}

\begin{proof}
The number of candidates for ordered partitions $\Pi$ on $V(\G)$ that induce the
edge profile $(P,\pi)$ on $\G$ is equal to the number of partitions of $V(\G)$ into
$q+1$ distinguishable cells such that the $i$-th cell has size $\pi(i)n$.
(Of course, we may assume that $(P,\pi)$ is a valid edge profile so that the numbers
$\pi(i) n$ and $P(i,j)nd$ are all integers.)
The number of such partitions is given by the multinomial term $\binom{n}{\pi(i)n; i \in \chi}$.

For any such ordered partition $\Pi$, the probability that its edge profile in $\G$ agrees with $(P,\pi)$ is

\[ \frac{\prod_{i} \binom{nd \pi(i)}{ndP(i,j); j \in \chi} \prod_{\{i,j\}: i \neq j}(ndP(i,j))! \prod_{i}(ndP(i,i)-1)!!}{(nd-1)!!}\,.\]

The first product counts, for every $0 \leq i \leq q$, the number of ways to partition the $nd \pi(i)$
half edges from $\Pi(i)$ into distinguishable sub-cells $\Pi(i,j)$ such that $\# \Pi(i,j) = ndP(i,j)$.
The half edges in $\Pi(i,j)$ are to be paired with half edges from $\Pi(j,i)$.
The second and third products count the number of ways to achieve these pairings.
As each configuration of pairings appear with probability $1/(nd-1)!!$, the formula above follows.

From the linearity of the expectation it follows that $\E{Z(P,\pi, \G)}$ equals
\begin{equation} \label{eqn:expectation}
\binom{n}{\pi(i)n; i \in \chi} \times \frac{\prod_{i} \binom{nd \pi(i)}{ndP(i,j); j \in \chi}
\prod_{\{i,j\}: i \neq j}(ndP(i,j))! \prod_{i}(ndP(i,i)-1)!!}{(nd-1)!!}.
\end{equation}

To analyze the asymptotic behaviour of (\ref{eqn:expectation}) we
use Stirling's approximation: $m! = \sqrt{2\pi m}(m/e)^m(1 + O(1/m))$.
Note that $(m-1)!! = \frac{m!}{2^{(m/2)}(m/2)!}$ for even $m$.
From Stirling's approximation we can easily conclude that
\begin{equation} \label{eqn:stirling}
\binom{m}{\rho_i m; 0 \leq i \leq q} = O(m^{-q/2})e^{mH(\rho_0, \ldots, \rho_q)} \quad \text{and}
\quad \frac{m!!}{\sqrt{m!}} = \Theta(m^{-1/4})\,.
\end{equation}
The $\rho_i m$ are non-negative integers such that $\rho_0 + \ldots + \rho_q = 1$.
The big O and big Theta are constants that are uniform over $m$.
From these two estimates we can simplify (\ref{eqn:expectation}). The polynomial factors in $n$
in the following come from the $m^{-q/2}$ and $m^{-1/4}$ factors in
the estimates above. The first two estimates below are a result of the first estimate in (\ref{eqn:stirling})
about the binomial coefficient. The last inequality comes from replacing double factorials
with the second estimate provided in (\ref{eqn:stirling}).

\begin{itemize}

\item $\binom{n}{\pi(i)n; i \in \chi} \leq \mathrm{poly}(n) \; e^{nH(\pi)}$,\\

\item $\prod_{i} \binom{nd \pi(i)}{ndP(i,j); j \in \chi} \leq
\mathrm{poly}(n) \; e^{nd\,\left [\sum_i \pi(i) H \left (\big \{\frac{P(i,j)}{\pi(i)}; \,0 \leq j \leq q \big \} \right) \right ]}$,\\

\item Finally,  $\frac{\prod_{\{i,j\}: i \neq j}(ndP(i,j))! \prod_{i}(ndP(i,i)-1)!!}{(nd-1)!!}$ is at most
$$\mathrm{poly}(n) \times \left [\frac{\prod_{(i,j): i \neq j}(ndP(i,j))! \prod_{i}(ndP(i,i))!}{(nd)!} \right ]^{1/2}.$$
Stirling's formula and algebraic simplifications imply that the right hand side of the third estimate above
is bounded from above by $\mathrm{poly}(n) \, e^{-\frac{nd}{2}\,H(P)}$.

\end{itemize}
The term $\sum_i \pi(i) H \big (\{ P(i,j)/\pi(i); \, 0 \leq j \leq q\} \big )$ is the conditional entropy of $P$ given $\pi$.
Algebraic simplification shows that it equals $H(P) - H(\pi)$. Consequently,
\[ \E{Z(P,\pi,\G)} \leq \mathrm{poly}(n) \, e^{n[H(\pi) + dH(P) - dH(\pi) - (d/2)H(P)]}\,,\]
and the above simplifies to the formula in the statement of the lemma.
\end{proof}

We now have the ingredients to prove a key inequality about FIID processes on trees.
\begin{thm} \label{thm:FIIDentropy}
Let $\Phi$ be an FIID process on $\T_d$ taking values in $\chi^{V(\T_d)}$. Let $(\rt, \rt')$ be
a directed edge of $\T_d$. Let $P_{\Phi}$ and $\pi_{\Phi}$ be the distributions of
$(\Phi(\rt), \Phi(\rt'))$ and $\Phi(\rt)$, respectively. Then
\begin{equation} \label{eqn:entropyinequality}
\frac{d}{2}H(P_{\Phi}) - (d-1)H(\pi_{\Phi}) \geq 0\,.
\end{equation}
\end{thm}
The term $\frac{d}{2} H(P_{\Phi}) - (d-1)H(\pi_{\Phi})$ will be denoted the \emph{entropy functional}.

\begin{proof}
For any two matrices $M$ and $M'$, set $|| M - M'|| = \max_{i,j} |M(i,j) - M'(i,j)|$.
For $0 < \eps < 1$ consider the event
$$A(\eps) = \big \{ \G \;\text{admits an ordered partition with edge profile}\; (P,\pi)
\; \text{satisfying}\; ||P - P_{\Phi}|| \leq \eps \big \}.$$
Observe that there are at most $(2nd)^{(q+1)^2}$ edge profiles $(P,\pi)$ that satisfy $ ||P - P_{\Phi}|| \leq \eps$.
Indeed, any such $P$ has entries $P(i,j)$ that are rational numbers of the form $a/nd$ with
$0 \leq a \leq (P_{\Phi}(i,j) + \eps)nd$. Since $P_{\Phi}(i,j) + \eps \leq 2$ for every entry $P_{\Phi}(i,j)$, there are no more than $2nd$
choices for $a$. As there are $(q+1)^2$ entries, the bound follows. Now taking a union bound over
all edge profiles $(P,\pi)$ satisfying $||P - P_{\Phi}|| \leq \eps$ and using Lemma \ref{lem:edgeprofile}
we conclude that
\begin{eqnarray*}
\pr{A(\eps)} &\leq& \sum_{(P,\pi): \, ||P - P_{\Phi}|| \leq \eps} \E{Z(P,\pi, \G)} \\
&\leq& \mathrm{poly}(n) \times \sup_{(P,\pi): \, || P - P_{\Phi}|| \leq \eps} e^{n \left [\frac{d}{2}H(P) - (d-1)H(\pi) \right ]}\,.
\end{eqnarray*}

The map $(P,\pi) \to (d/2)H(P) - (d-1)H(\pi)$ is continuous with respect to the norm $|| \cdot ||$.
Continuity implies that there exists $\delta(\eps) \to 0$ as $\eps \to 0$ such that
$\frac{d}{2}H(P) - (d-1)H(\pi) \leq \frac{d}{2}H(P_{\Phi}) - (d-1)H(\pi_{\Phi}) + \delta(\eps)$
if $|| P - P_{\Phi}|| \leq \eps$. As a result,
\begin{equation} \label{eqn:upbound}
\pr{A(\eps)} \leq \mathrm{poly}(n) \times e^{n[\frac{d}{2}H(P_{\Phi}) - (d-1)H(\pi_{\Phi}) + \delta(\eps)]}\,.
\end{equation}

If $\frac{d}{2}H(P_{\Phi}) - (d-1)H(\pi_{\Phi}) < 0$ then by choosing $\eps$ small enough
we can ensure that $\frac{d}{2}H(P_{\Phi}) - (d-1)H(\pi_{\Phi}) + \delta(\eps) < 0$.
This implies from (\ref{eqn:upbound}) that $\pr{A(\eps)} \to 0$ as $n \to \infty$. However,
this is a contradiction because $\pr{A(\eps)} \to 1$ as $n \to \infty$ due to the concentration
inequality in (\ref{eqn:concentration}). Indeed, for any $\eps > 0$ we have from (\ref{eqn:concentration}) that
$$\pr{A(\eps)} \geq \pr{||P_{\Phi_{\G}} - P_{\Phi}|| \leq \eps} \geq
1 - e^{-\frac{\eps^2 n}{C_{d,r,q}}} - \mathrm{err}_n\,.$$
\end{proof}

\section{Proof of Theorem \ref{thm:corr}} \label{sec:corrprf}

We prove Theorem \ref{thm:corr} before Theorem \ref{thm:main}
because it serves as an introduction to the ideas used to prove
the latter, which are of a more technical nature.
We begin with the proof of the upper bound.

Let $\{\prc_d\}$ be a sequence of FIID percolation processes with $\prc_d$ defined on $\T_d$.
Let $P_{\prc_d}$ be the $2\times 2$ matrix whose entries are $\pr{\prc_d(\rt) = i, \prc_d(\rt')=j}$,
where $(\rt, \rt')$ is a directed edge of $\T_d$ and $i,j \in \{0,1\}$. This is the edge profile of
$\prc_d$, which we can express in terms of the correlation ratio and density.
For convenience we write $\ro = \corr{\prc_d}$ and $\al = \den{\prc_d}$. We have
$$ P_{\prc_d} = \left [ \begin{array}{cc}
1 - 2\al + \ro \al^2 & \al - \ro \al^2 \\
\al - \ro \al^2 & \ro \al^2
\end{array} \right ]\,.$$

Set $h(x) = -x\log x$ for $0 \leq x \leq 1$ with the convention that $h(0) = 0$.
The entropy functional is
\begin{align} \label{eqn:entropyfirstorder}
&\frac{d}{2} H(P_{\prc_d}) - (d-1)H(\pi_{\prc_d}) =\\
\nonumber &\frac{d}{2} \left [h(\ro \al^2) + 2h(\al - \ro \al^2) + h(1 - 2\al + \ro \al^2)]
- (d-1)[h(\al) + h(1 - \al) \right ].
\end{align}
We will first derive an upper bound to the entropy functional and then use the non-negativity
of the entropy functional from Theorem \ref{thm:FIIDentropy} to get a bound on $\al$.
We will use the following properties of $h$.
\begin{lem} \label{lem:hproperty}
The function $h(x) = -x\log x$ for $x > 0$ and $h(0) = 0$ satisfies the following.
\begin{align}
\label{eqn:h1}  h(xy) &= xh(y) + yh(x) \;\;\text{for all}\;\; x,y \geq 0.\\
\label{eqn:h2} x - \frac{x^2 + x^3}{2} &\leq h(1-x) \leq x - \frac{x^2}{2} \;\;\text{for all}\;\; 0 \leq x \leq 1.
\end{align}
\end{lem}

\begin{proof}
The identity in (\ref{eqn:h1}) follows from definition.
The inequalities in (\ref{eqn:h2}) follow from Taylor expansion. They are valid for $x = 1$.
For $0 \leq x < 1$, note that $-\log(1-x) = \sum_k x^k/k$. Hence, $-\log(1-x) \geq x + x^2/2$,
which implies that $h(1-x) \geq x - (1/2)x^2 - (1/2)x^3$. Similarly,
$-\log(1-x) \leq x + (1/2)x^2 + (1/3)x^3 (1 + x + x^2 \cdots)$, which shows that
$-\log(1-x) \leq x + (1/2)x^2 + x^3/(3(1-x))$ for $0 \leq x < 1$. Consequently,
$h(1-x) \leq x - (1/2)x^2 - (1/6)x^3 \leq  x - (1/2)x^2$.
\end{proof}

\begin{lem} \label{lem:Hupperbound}
The entropy function for $\prc_d$ satisfies the following upper bound.
\begin{equation} \label{eqn:Hupperbound}
\frac{d}{2}\,H(P_{\prc_d}) - (d-1)H(\pi_{\prc_d}) \leq -\frac{d}{2}\Psi(\ro) \al^2 + \al + h(\al) + \frac{(2\ro + 1)d}{2} \al^3.
\end{equation}
\end{lem}

\begin{proof}
The entropy functional is the r.h.s.~of (\ref{eqn:entropyfirstorder}), which we analyze term by term.
The following inequality is derived by applying (\ref{eqn:h1}) for the equality and
then using that $h(1-x) \leq x$ from (\ref{eqn:h2}) with $x = \ro \al$ for the inequality.
\begin{eqnarray*}
h(\ro \al^2) + 2h(\al - \ro \al^2) &=& h(\ro)\al^2 + 2\al h(1-\ro \al) + 2h(\al) \\
& \leq & [h(\ro) + 2\ro]\al^2 + 2h(\al).
\end{eqnarray*}
Similarly, the upper bound to $h(1-x)$ from (\ref{eqn:h2}) with $x = 2\al - \ro \al^2$ gives
$$h(1 - 2\al + \ro \al^2) \leq 2\al - \ro \al^2 - \frac{1}{2}(2\al - \ro \al^2)^2\,.$$
Simplification then implies
$$h(\ro \al^2) + 2h(\al - \ro \al^2) + h(1 - 2\al + \ro \al^2) \leq [h(\ro) + \ro -2]\al^2 + 2[\al + h(\al) +\ro \al^3].$$
Note that $h(\ro) + \ro -1 = - \Psi(\ro)$, if we recall that $\Psi(c) = c\log c - c + 1$.
Also, the lower bound to $h(1-x)$ from (\ref{eqn:h2}) gives $h(1-\al) \geq \al -(1/2)\al^2 - (1/2)\al^3$.
Therefore, the r.h.s.~of (\ref{eqn:entropyfirstorder}) satisfies
\begin{align*} \frac{d}{2} \left [h(\ro \al^2) + 2h(\al - \ro \al^2) + h(1 - 2\al + \ro \al^2)]
- (d-1)[h(\al) + h(1 - \al) \right ] &\leq \\
-\frac{d}{2}\Psi(\ro) \al^2 + \al + h(\al) + \frac{(2\ro + 1)d}{2} \al^3&.
\end{align*}
\end{proof}

The top order terms from the r.h.s.~of (\ref{eqn:Hupperbound}) are $-(d/2)\Psi(\ro) \al^2$ and $h(\al)$,
which are of the same order only if $\al$ is of order $(\log d)/d$. Writing $\al = \beta_d \frac{\log d}{d}$
the inequality (\ref{eqn:Hupperbound}) translates to
$$ \frac{d}{2}\,H(P_{\prc_d}) - (d-1)H(\pi_{\prc_d}) \leq
\left ( -\frac{\Psi(\ro)}{2}\beta_d + 1 + \frac{1-\log \log d - \log \beta_d}{\log d} +
\frac{(2\ro + 1)\log d}{2d}\beta_d^2 \right ) \frac{\beta_d \log^2 d}{d}.$$
From the non-negativity of the entropy functional in Theorem \ref{thm:FIIDentropy} and the inequality
above we conclude that
\begin{equation} \label{eqn:Hupbound}
-\frac{\Psi(\ro)}{2}\beta_d + 1 + \frac{1-\log \log d - \log \beta_d}{\log d} +
\frac{(2\ro + 1)\log d}{2d}\beta_d^2 \geq 0.
\end{equation}

Recall that in order to prove Theorem \ref{thm:corr} we must show that for any $\eps > 0$
there is a $d_0$ such that if $d \geq d_0$ then
$\beta_d \leq \frac{2}{\Psi(c)} \big(1 - \frac{\log\log d -1-\log \Psi(c) -\eps}{\log d} \big).$
If $\beta_d \leq 1/\Psi(c)$ then this conclusion is true for all large $d$ even for $\eps = 0$
so long as $c \neq 1$ (in which case $\Psi(c) = 0$). Therefore, we may assume that $\beta_d \geq 1/\Psi(c)$
for every $d$. If $\beta_d \geq 1/\Psi(c)$ then the inequality in (\ref{eqn:Hupbound}) implies that
$$-\frac{\Psi(\ro)}{2}\beta_d + 1 + \frac{1-\log \log d + \log \Psi(c)}{\log d} + \frac{(2\ro + 1)\log d}{2d}\beta_d^2 \geq 0$$
because $\log \beta_d \geq - \log \Psi(c)$. Hence, we must have that $q(\beta_d) \geq 0$, where
$$q(x) =  \frac{(2\ro + 1)\log d}{2d}x^2 -\frac{\Psi(\ro)}{2}x + 1 + \frac{1 + \log \Psi(c) -\log\log d}{\log d}.$$

In the following we prove that under the assumptions of Theorem \ref{thm:corr} $\beta_d$ is at most
the smaller root of the quadratic polynomial $q$, provided that $d$ is sufficiently large. Then we
provide an upper bound for the smaller root of $q$ to conclude the proof of the theorem.

The two roots of $q$ are
\begin{equation} \label{eqn:roots}
\frac{d}{(2\ro + 1)\log d} \left ( \frac{\Psi(\ro)}{2} \pm \sqrt{\frac{\Psi^2(\ro)}{4} -
\frac{(4\ro + 2)}{d}(\log d - \log\log d + \log \Psi(c)+1)} \right).
\end{equation}
We first show that $q$ has two real roots for sufficiently large $d$.
From Lemma \ref{lem:Psibound} below we have that
$$\frac{\Psi(\ro)^2}{4\ro +2} \geq \frac{\Psi(c)^2}{6} \;\;\text{if}\;\; \ro \leq c < 1,\;\;
\text{and}\;\; \frac{\Psi(\ro)^2}{4\ro +2} \geq \frac{\Psi(c)^2}{6c} \;\;\text{if}\;\; \ro \geq c > 1.$$
Consequently, for both the cases $\ro \leq c <1$ or $\ro \geq c > 1$ there exists
a $d_1 = d_1(c)$ such that for $d \geq d_1$,
$$\frac{\Psi^2(\ro)}{4} - \frac{4\ro + 2}{d}(\log d - \log\log d + \log \Psi(c)+ 1) > 0\;\;
\text{as}\;\; \frac{\log d - \log\log d + \log \Psi(c) + 1}{d} \to 0.$$
The quadratic $q$ has two real roots when the inequality above holds because of the formula for the roots
of $q$ given in (\ref{eqn:roots}).

Assuming for $d \geq d_1$, so that $q$ has two real roots, we now show that $\beta_d$ is not
larger than the smaller root of $q$ for all large $d$. Lemma \ref{lem:densityvanish} below implies that
$\beta_d = o(d/\log d)$ as $d \to \infty$. Therefore, given $c$ there exists $d_2$ such that for
$d \geq d_2$ we have $\beta_d \leq \min \{\Psi(c), \Psi(c)/c \}\,\frac{d}{9\log d}$.
But note from (\ref{eqn:roots}) that the larger root of $q$ is at least
$\frac{d}{(2\ro + 1)\log d} \cdot \frac{\Psi(\ro)}{2}$, and this quantity is at least
$\min \{ \Psi(c), \Psi(c)/c\}\, \frac{d}{6\log d}$ due to the lower bound on $\Psi(\ro)/(2\ro +1)$ provided
in Lemma \ref{lem:Psibound} below. Therefore, $\beta_d$ can not be bigger than the larger root of $q$.
Then $\beta_d$ is bounded from above by the smaller root as $q$ is negative between its two roots.

We have deduced that for $d \geq \max \{d_1,d_2\}$, $\beta_d$ is bounded from above
by the smaller root of $q$. An elementary simplification shows that the smaller root equals
$$ \frac{2(1 - \frac{\log\log d}{\log d} + \frac{\log \Psi(c)}{\log d} + \frac{1}{\log d})}{\frac{\Psi(\ro)}{2} +
\sqrt{\frac{\Psi^2(\ro)}{4} - \frac{(4\ro + 2)}{d}(\log d - \log\log d + \log \Psi(c) +1)}}\,.$$
Using $\sqrt{1-x} \geq 1- x$ for $0 \leq x \leq 1$, the above may be bounded to show that for $d \geq \max \{d_1,d_2\}$,
$$ \beta_d \leq \frac{2}{\Psi(\ro)} \left ( 1 + \frac{1 + \log \Psi(c) - \log \log d}{\log d} +
O\big (\frac{(2\ro +1 )\log d}{\Psi^2(\ro)d} \big ) \right).$$
The upper bound of Theorem \ref{thm:corr} now follows because the term
$\frac{(2\ro +1 )\log d}{\Psi^2(\ro)d}$ is of order $O(\log d /d)$ if $\ro$ is uniformly bounded away from 1 in $d$.
This follows from the lower bound on $\Psi(\ro)$ given in Lemma \ref{lem:Psibound} below.

\begin{lem} \label{lem:Psibound}
The function $\Psi$ satisfies the following lower bounds.
\begin{itemize}
\item $\Psi(\ro) \geq \Psi(c)$ if $\ro \leq c < 1$ or if $\ro \geq c > 1$.
\item If $\ro \leq c < 1$ then $\Psi(\ro)/(2\ro+1) \geq \Psi(c)/3$.
\item If $\ro \geq c > 1$ then $\Psi(\ro)/(2\ro+1) \geq \Psi(c)/3c$.
\end{itemize}
\end{lem}

\begin{proof}
Differentiating $\Psi(x)$ shows that it is decreasing for $x \in [0,1]$ and
increasing for $x \in [1,\infty)$. This implies the claim in the first bullet.
Now, suppose that $\ro \leq c < 1$. Then $\Psi(\ro) \geq \Psi(c)$ and $2\ro + 1 \leq 3$,
so $\Psi(\ro)/(2\ro + 1) \geq \Psi(c)/3$, as required.
Next, suppose that $\ro \geq c > 1$. Then $2\ro + 1 \leq 3 \ro$ and so $\Psi(\ro)/(2\ro + 1) \geq \Psi(\ro)/(3\ro)$.
Observe that $\Psi(\ro)/\ro = \log \ro - 1 + 1/\ro$ and the function $x \to \log x - 1 + 1/x$ is
increasing if $x \geq 1$. Hence, $\Psi(\ro)/ (2\ro +1) \geq \Psi(c)/3c$ if $\ro \geq c$.
\end{proof}

\begin{lem} \label{lem:densityvanish}
Suppose $\ro = \corr{\prc_d}$ is bounded away from 1 uniformly in $d$, i.e.,
$\ro \notin (1 - \delta, 1+ \delta)$ for all $d$ and some $\delta > 0$.
Then $\al = \den{\prc_d} \to 0$ as $d \to \infty$. Equivalently, $\beta_d = o(\log d /d)$
as $d \to \infty$.
\end{lem}

\begin{proof}
Suppose otherwise, for the sake of a contradiction. By
moving to a subsequence in $d$ we may assume that
$\al \to \alpha_{\infty} > 0$ and $\ro \to \rho_{\infty} \neq 1$.
Then dividing (\ref{eqn:entropyfirstorder}) through by $d$ and
taking the limit in $d$ we see that
$$\frac{1}{2} \left [h(\rho_{\infty} \alpha_{\infty}^2) + 2h(\alpha_{\infty} - \rho_{\infty}\alpha_{\infty}^2)
+ h(1 - 2\alpha_{\infty} + \rho_{\infty}\alpha_{\infty}^2) \right ] -h(\alpha_{\infty}) - h(1-\alpha_{\infty}) \geq 0\,.$$

It is easy to show by differentiating with respect to $\rho_{\infty}$ that the l.h.s.~of
the above is uniquely maximized at $\rho_{\infty} = 1$, for any fixed $\alpha_{\infty} > 0$. However,
when $\rho_{\infty} = 1$, the three summands involving $\rho_{\infty}$ above add up to give
$2[h(\alpha_{\infty}) + h(1-\alpha_{\infty})]$. (This conclusion, of course, follows from subadditivity
of entropy and the fact that the entropy of a pair of distributions is maximal when they are
independent.) Thus, when $\rho_{\infty} = 1$ the l.h.s.~of the above inequality equals zero, but is otherwise
negative for $\rho_{\infty} \neq 1$ and fixed $\alpha_{\infty} > 0$. This contradicts the inequality above
and allows us to conclude that $\al \to 0$.
\end{proof}

\subsection{Construction of large density percolation processes with given correlation ratio} \label{sec:corrlwbdd}

To construct processes as in the second statement of Theorem \ref{thm:corr} we interpolate between Bernoulli
site percolation on $\T_d$ and FIID independent sets of large density. Recall that there exist FIID
independent sets $\mathbf{I}_d$ on $\T_d$ with density $(1-o(1))\frac{\log d}{d}$ as $d \to \infty$ \cite{LW}.
To construct $\mathbf{Z}_d$ from Theorem \ref{thm:corr} we
first fix parameters $0 < p < 1$ and $x > 0$. Using two independent sources of random
labellings of $\T_d$, we generate $\mathbf{I}_d$ and a Bernoulli percolation,
$\ber$, having density $x \frac{\log d}{d}$. Then for each vertex $v$ we
define $\mathbf{Z}_d(v) = \mathbf{I}_d(v)$ with probability $p$, or $\mathbf{Z}_d(v) = \ber(v)$
with probability $1-p$. The decisions to choose between $\mathbf{I}_d$ or $\ber$ are made
independently between the vertices. Then $\mathbf{Z}_d$ is an FIID percolation on $\T_d$.
Since $\corr{\mathbf{I}_d} = 0$ and $\corr{\ber} = 1$, it is easy to calculate that

\begin{eqnarray*}
\den{\mathbf{Z}_d} &=& p \, \den{\mathbf{I}_d} + (1-p) \, \den{\ber}, \;\text{and}\\
\corr{\mathbf{Z}_d} &=& \frac{(1-p)^2\den{\ber}^2 + 2p(1-p)\den{\mathbf{I}_d}\den{\ber}}{\den{\mathbf{Z}_d}^2}\,.
\end{eqnarray*}

In the following calculation we ignore the $1-o(1)$ factor from $\den{\mathbf{I}_d}$ for tidiness.
It does not affect the conclusion and introduces the $o(1)$ term in the statement of the
lower bound for $\den{\mathbf{Z}}$ from Theorem \ref{thm:corr}.
Continuing from the above we see that
$\den{\mathbf{Z}_d} = (p + (1-p)x)\frac{\log d}{d}$ and
$\corr{\mathbf{Z}_d} = \frac{(1-p)^2x^2 + 2p(1-p)x}{(p + (1-p)x)^2}$.
Setting the correlation ratio equal to $c$ and solving for $x$ in terms of $p$ gives
$x = \frac{p}{1-p} (\frac{1}{\sqrt{1-c}} - 1)$. Consequently, the density of $\mathbf{Z}_d$
is $\frac{p}{\sqrt{1-c}} \cdot \frac{\log d}{d}$. By letting $p \to 1$ we deduce the required conclusion.

\section{Proof of Theorem \ref{thm:main}} \label{sec:mainprf}

Let $\prc_d$ be a sequence of FIID percolation processes with $\prc_d$ defined on $\T_d$
and the correlation ratios $\corr{\prc_d} \to 0$ as $d \to \infty$. We may conclude from
Theorem \ref{thm:corr} that $\limsup_{d \to \infty} \frac{\den{\prc_d}}{(\log d)/d} \leq 2$.
In order to prove Theorem \ref{thm:main}, where the constant 2 is replace by 1,
we also use the entropy inequality in Theorem \ref{thm:FIIDentropy}, but we apply
it to many copies of a given percolation process $\prc_d$, where the copies are coupled in a
particular way. In turns out that the entropy inequality provides better bounds when it is applied
to many copies of $\prc_d$. This important idea is borrowed from statistical physics where it
was initially observed through the \emph{overlap structure} of the Sherrington-Kirkpatrick model \cite{Tal}.

To derive bounds on $\den{\prc_d}$ from the entropy inequality applied to
several copies of $\prc_d$ we need to find a upper bound to the general
entropy functional in terms of the density of each colour class (the analogue of
Lemma \ref{lem:Hupperbound}). This bound follows from an important convexity argument.
We begin with the derivation of this upper bound. Once it is established, which is presented in
Lemma \ref{lem:maxentropy}, the remainder of the proof follows the same argument
as what has been done for independent sets in \cite[Section 2.1]{RV}.

\subsection{Upper bound for the entropy functional} \label{sec:entropybound}

The subadditivity of entropy implies that for any edge profile $(P,\pi)$ the entropy
$H(P) \leq 2H(\pi)$, which in turn implies that $(d/2)H(P) - (d-1)H(\pi) \leq H(\pi)$.
However, this is useless towards analyzing the entropy inequality since $H(\pi) \geq 0$.
In fact, this upper bound is sharp if and only if $P(i,j) = \pi(i) \pi(j)$ for every pair $(i,j)$.
However, a high density percolation process $\prc_d$ whose correlation ratio satisfies
the hypothesis of Theorem \ref{thm:main} cannot have an edge profile with this property.
In order to derive a suitable upper bound we have to bound the entropy
functional subject to constraints on the edge profile induced by small correlation ratios.

We derive such a bound from convexity arguments that take into account
the constraints put forth by small correlation ratios. Although our bound is not sharp, it appears
to become sharp as the degree $d \to \infty$. This is the reason as to why we get
an asymptotically optimal bound on the density as $d \to \infty$. When the correlation ratios do
not converge to zero the bound becomes far from sharp and we cannot get bounds on 
the density that improve upon Theorem \ref{thm:corr} in a general setting.

Before proceeding with the calculations we introduce some notation and provide the setup.
Let $(P,\pi)$ be an edge profile such that the entries of $P$ are indexed by $\chi^2$, where $\# \chi = q + 1 \geq 2$.
Suppose that there is a subset $\Lambda \subset \chi^2$ (which represents the pairs $(i,j)$ where $P(i,j)$ is
very small) with the following properties.

\begin{enumerate}
\item $\Lambda$ is symmetric, that is, if $(i,j) \in \Lambda$ then $(j,i) \in \Lambda$.

\item There exists an element $0 \in \chi$ such that $\Lambda$ does not contain
any of the pairs $(0,j)$ for every $j \in \chi$.

\item There are positive constants $K \leq 1/(eq)$ and $J$ such that $P(i,j) \leq K$
for every $(i,j) \in \Lambda$ and $\pi(i) \leq J$ for every $i \neq 0$ ($e = 2.78\ldots$).
\end{enumerate}

For $i \in \chi$, set $\Lm_i = \{ j \in \chi : (i,j) \in \Lm \}$ and define the quantities
\begin{align*}
\pi(\Lm_i) &= \sum_{j \in \Lm_i} \pi(j) \\
P(\Lm_i) &= \sum_{j \in \Lm_i} P(i,j) \\
\pi^2(\Lm) &= \sum_i \pi(i) \pi(\Lm_i) = \sum_{(i,j) \in \Lm} \pi(i) \pi(j) \,.
\end{align*}

In our application in Section \ref{sec:mainproof} both $K$ and $J$ will converge to zero as
$d \to \infty$ while $\chi$ will remain fixed. So the bound stipulated on $K$ will be satisfied for large $d$.

\begin{lem} \label{lem:maxentropy}
With the setup as above the following inequality holds for every $d$:
$$\frac{d}{2}H(P)  - (d-1)H(\pi) \leq H(\pi) - \frac{d}{2}\pi^2(\Lm) + q^2 \left( dK + dK \log \left (\frac{J^2}{K} \right ) \right ).$$
\end{lem}

\begin{proof}
We use the concavity of the function $h(x) = -x \log x$.
Using the identity $h(xy) = xh(y) + yh(x)$ from Lemma \ref{lem:hproperty} we get
\[H(P) = \sum_{(i,j)} h \left (\pi(i)\cdot \frac{P(i,j)}{\pi(i)} \right ) = \sum_{(i,j)} \pi(i) h \left (\frac{P(i,j)}{\pi(i)} \right ) + H(\pi)\,.\]

To bound $\sum_i \pi(i) h(\frac{P(i,j)}{\pi(i)})$ we consider the
summands for $i \in \Lm_j$ and $i \notin \Lm_j$ separately.
Jensen's inequality applied to $h(x)$ implies that
\begin{align} \label{eqn:convexitybdd}
\sum_{i \in \Lm_j} \pi(i)\, h \left (\frac{P(i,j)}{\pi(i)} \right ) \leq \pi(\Lm_j) \, h \left (\frac{P(\Lm_j)}{\pi(\Lm_j)} \right ), \; \text{and} \\
\label{eqn:convexitybdd2}
\sum_{i \notin \Lm_j} \pi(i)\,h \left (\frac{P(i,j)}{\pi(i)} \right ) \leq (1-\pi(\Lm_j))\, h \left (\frac{\pi(j) - P(\Lm_j)}{1-\pi(\Lm_j)} \right )\,.
\end{align}

Since $\pi(\Lm_0) = P(\Lm_0) = 0$, when $j=0$ the sum in (\ref{eqn:convexitybdd}) is empty
the sum in (\ref{eqn:convexitybdd2}) equals $h(\pi(0))$. We analyze the term on the r.h.s.~of (\ref{eqn:convexitybdd2})
when $j \neq 0$.
\begin{align*}
(1-\pi(\Lm_j))\, h \left (\frac{\pi(j) - P(\Lm_j)}{1-\pi(\Lm_j)} \right ) =& - (\pi(j) - P(\Lm_j)) \log \left ( \frac{\pi(j) - P(\Lm_j)}{1-\pi(\Lm_j)} \right )\\
=& \, h(\pi(j) - P(\Lm_j)) + (\pi(j) - P(\Lm_j)) \log (1- \pi(\Lm_j)).
\end{align*}

Notice that the term
$$h(\pi(j)-P(\Lm_j)) = h(\pi(j) \cdot \left (1-\frac{P(\Lm_j)}{\pi(j)}) \right ) = \left(1-\frac{P(\Lm_j)}{\pi(j)} \right)h(\pi(j))
+ \pi(j)\,h \left (1-\frac{P(\Lm_j)}{\pi(j)} \right ).$$
The first of the two summands on the very right, $(1-\frac{P(\Lm_j)}{\pi(j)})\,h(\pi(j))$,
equals $h(\pi(j)) + P(\Lm_j) \log (\pi(j))$. The second,  $\pi(j)h \big (1-\frac{P(\Lm_j)}{\pi(j)}\big)$, satisfies
$\pi(j)h \big (1-\frac{P(\Lm_j)}{\pi(j)} \big) \leq P(\Lm_j)$ because $h(1-x) \leq x$ from (\ref{eqn:h2}). Thus,
$$h \big (\pi(j)-P(\Lm_j) \big) \leq h(\pi(j)) + P(\Lm_j) \log (\pi(j)) + P(\Lm_j)\;\; \text{when}\;\; j \neq 0.$$

Now consider $(\pi(j) - P(\Lm_j)) \log (1- \pi(\Lm_j))$. Since $\log(1-x) \leq -x$ for $0 \leq x \leq 1$,
we have that $\log (1- \pi(\Lm_j)) \leq - \pi(\Lm_j)$. As a result,
$$\left (\pi(j) - P(\Lm_j) \right ) \log (1- \pi(\Lm_j)) \leq - \pi(j) \pi(\Lm_j) + P(\Lm_j).$$

We may now conclude that for $j \neq 0$,
\begin{equation} \label{eqn:complicated}
(1-\pi(\Lm_j))\, h \left (\frac{\pi(j) - P(\Lm_j)}{1-\pi(\Lm_j)} \right ) \leq h(\pi(j)) - \pi(j) \pi(\Lm_j)
 + P(\Lm_j) \log(\pi(j)) + 2P(\Lm_j).
\end{equation}

We may also simplify the r.h.s.~of (\ref{eqn:convexitybdd}):
$\pi(\Lm_j) h \big (\frac{P(\Lm_j)}{\pi(\Lm_j)} \big ) = P(\Lm_j) \log \big ( \frac{\pi(\Lm_j)}{P(\Lm_j)} \big)$.
Therefore, we deduce from (\ref{eqn:convexitybdd}), (\ref{eqn:convexitybdd2}) and (\ref{eqn:complicated})
that if $j \neq 0$ then
\begin{equation} \label{eqn:intermediatebdd}
\sum_i \pi(i) h \left (\frac{P(i,j)}{\pi(i)} \right ) \leq h(\pi(j)) - \pi(j) \pi(\Lm_j)
 + 2P(\Lm_j) + P(\Lm_j) \log \left (\frac{\pi(j) \pi(\Lm_j)}{P(\Lm_j)} \right).
\end{equation}

The last two summands on the r.h.s.~of (\ref{eqn:intermediatebdd}) contribute to the big O error term involving $J$ and $K$,
which we now demonstrate. We have that the second last summand $P(\Lm_j) \leq qK$ by hypothesis (3).
Consider the last summand in (\ref{eqn:intermediatebdd}). We have that $\pi(j) \leq J$ for $j \neq 0$ by hypothesis (3).
Moreover, $\pi(\Lm_j) \leq qJ$ for $j \neq 0$ because $0 \notin \Lm_j$ by hypothesis (2), so each $\pi(i) \leq J$
for $i \in \Lm_j$ by hypothesis $(3)$, and $\# \Lm_j \leq q$. Hence, $\pi(j) \pi(\Lm_j) \leq qJ^2$ for $j \neq 0$.
The function $h(x)$ is increasing for $x \leq 1/e$. Hence, as $P(\Lm_j) \leq qK \leq 1/e$ by hypothesis (3), we have
\begin{eqnarray*}
P(\Lm_j) \log \left (\frac{\pi(j) \pi(\Lm_j)}{P(\Lm_j)} \right ) &=& h(P(\Lm_j)) + P(\Lm_j) \log \big (\pi(j) \pi(\Lm_j) \big ) \\
&\leq& h(qK) + qK \log (qJ^2)\\
&=& qK \log \left (\frac{J^2}{K} \right )\,.
\end{eqnarray*}
Applying these bounds to the last two summands on the r.h.s.~of (\ref{eqn:intermediatebdd})
we get that that for $j \neq 0$,
$$ \sum_i \pi(i) h \left (\frac{P(i,j)}{\pi(i)} \right) \leq h(\pi(j)) - \pi(j)\pi(\Lm_j) + 2qK + qK \log \left (\frac{J^2}{K} \right).$$

Summing over all $j$ now gives
\begin{eqnarray*}
\sum_{(i,j)} \pi(i) h \left (\frac{P(i,j)}{\pi(i)} \right) &\leq& 
h(\pi(0)) + \sum_{j \neq 0} [h(\pi(j)) - \pi(j) \pi(\Lm_j)] + 2q^2 \left ( K + K \log \left (\frac{J^2}{K} \right) \right)\\
&=& H(\pi) - \pi^2(\Lm) + 2q^2 \left ( K + K \log \left (\frac{J^2}{K} \right) \right)\,.
\end{eqnarray*}
Consequently, $H(P) \leq 2H(\pi) - \pi^2(\Lm) + 2q^2 \big ( K + K \log(J^2/K) \big)$, and this implies the lemma.
\end{proof}

\subsection{Tools and setup of the proof} \label{sec:coupling}

In this section we set up important tools and terminology that we will
use to prove Theorem \ref{thm:main} in the following section.

\subsubsection{Many overlapping processes from a given percolation process} \label{sec:overlap}

To prove Theorem \ref{thm:main} we construct a coupled sequence of
exchangeable percolation processes $ \prc^0_d, \prc^1_d, \ldots$ from a given percolation
process $\prc_d$. We then apply the entropy inequality to the FIID process
$(\prc^1_d, \ldots, \prc^{k}_d)$ to derive bounds on $\den{\prc_d}$ that improve
upon Theorem \ref{thm:corr} by taking $k$ to be arbitrarily large.
The process $\prc^i_d$ is defined by varying the random labels that serve
as input to the factor of $\prc_d$. The labels are varied by first generating a random
subset of vertices of $\T_d$ via a Bernoulli percolation and then repeatedly re-randomizing
the labels over this subset to produce the random labelling for the $\prc^i_d$.
The processes are coupled since their inputs agree on the complement of this random subset.
The details follow.

Let $\iid^0_d, \iid^1_d, \iid^2_d, \ldots$ be independent
random labellings of $\T_d$. Fix a parameter $0 \leq p \leq 1$, and let
$\ber^{p}$ be a Bernoulli percolation on $\T_d$ having density $p$.
If $f_{\prc_d}$ is the factor associated to $\prc_d$ then the FIID percolation process
$\prc^i_d$ has factor $f_{\prc_d}$ but the random labelling used for it, say
$\cplbl^i = (W^i(v), v \in V(\T_d))$, is defined by $W^i(v) = X^i_d(v)$ if $\ber^p(v) = 1$ and
$W^i(v) = X^0(v)$ if $\ber^p(v) = 0$.

For $k \geq 1$ consider the FIID process $(\prc_d^1,\ldots,\prc^k_d)$, which takes
values in $\chi^{V(\T_d)}$ where $\chi = \{0,1\}^k$. By identifying elements of $\{0,1\}^k$ with
subsets of $\{1,\ldots,k\}$ the edge profile of this process can be described as follows.
Let $(\rt,\rt')$ be a fixed edge of $\T_d$. For subsets $S, T \subset \{1, \ldots, k\}$,

\begin{eqnarray*}
\pi(S) &=& \pr{\prc^i_d(\rt) = 1 \;\text{for}\; i \in S \;\text{and}\; \prc^i_d(\rt) = 0 \;\text{for}\; i \notin S} \\
P(S,T) &=& \mathbb{P} \Big [ \prc^i_d(\rt) = 1, \prc^j_d(\rt') = 1 \;\text{for}\; i \in S, j \in T \\
&\;&\text{and}\; \prc^i_d(\rt) = 0, \prc^j_d(\rt') = 0 \;\text{for}\; i \notin S, j \notin T \Big ]
\end{eqnarray*}

Observe that $\pi(\{i\}) \leq \den{\prc_d}$ for every $i$, and if $S \cap T \neq \emptyset$ then
for any fixed $i \in S \cap T$ we have
\begin{equation} \label{eqn:pstbound}
P(S,T) \leq \pr{\prc^i_d(\rt) = 1, \prc^i_d(\rt') = 1} \leq \corr{\prc_d}\, \den{\prc_d}^2.
\end{equation}

We describe some important properties of the coupled processes that will be used to
complete the proof of Theorem \ref{thm:main}. Since the statement of the theorem
is about the scaled density $\pi(\{1\}) \frac{d}{\log d}$, we find a
probabilistic interpretation for these scaled quantities. We use the coupling
to find such an interpretation for the ratio $\pr{\prc^i_d(\rt) = 1 \;\text{for every}\; i \in S}/\den{\prc_d}$.
Define the normalized density $\alpha(S)$ for any non empty and finite set $S \subset \{1,2,3, \ldots \}$ by
\begin{equation} \label{eqn:alpha}
\alpha(S) \frac{\log d}{d} = \pr{\prc^i_d(\rt) = 1 \;\text{for every}\; i \in S}\,.
\end{equation}

Since the coupled percolation processes are exchangeable, $\alpha(S) = \alpha(\{1, \ldots, \# S\})$ for all $S$.
For convenience, we write $\alpha_{i,d,p} = \alpha(\{1,\ldots, i\})$ and call these the
\emph{intersection densities} of the processes $\prc^1_d, \prc^2_d, \ldots$
(although, these densities are normalized by the factor of $(\log d)/d$).
When we apply the entropy inequality to $(\prc^1_d, \ldots, \prc^{k}_d)$
we will get an expression in terms of the intersection densities of the $k$ percolation processes.
In order to analyze this expression we have to realize the ratios $\alpha_{i,d,p}/\alpha_{1,d}$ as the moments
of a random variable $Q_{d,p}$. We define $Q_{d,p}$ in the following.

\subsubsection{The stability variable} \label{sec:stability}

The random variable $Q_{d,p}$ is defined on a new probability space, which is obtained from the
original probability space generated by the random labels $\iid_d$ by essentially restricting
to the support of the factor $f_{\prc_d}$. Formally, the new sample space is the set $\{ \prc^0_d(\rt) \equiv 1\}$
considered as a subset of the joint sample space of $\iid_d^0, \iid_d^1, \ldots$, and $\ber^p$.
The new $\sigma$-algebra is the restriction of the $\sigma$-algebra generated by $\iid_d^0, \iid_d^1,\ldots$,
and $\ber^{p}$ to $\{\prc^0_d(\rt) \equiv 1\}$. The new expectation operator $\mathbb{E}^*$ is defined by
\[ \Est{U} = \frac{\E{\prc^0_d(\rt) \,U}}{\E{\prc^0_d(\rt)}} \]
for any measurable random variable $U$ defined on $\{ \prc^0_d(\rt) \equiv 1\}$.
In the following, we write $\prc^i$ to stand for $\prc^i_d$.

\begin{lem} \label{lem:restriction}
If $\mathcal{F}$ is a $\sigma$-algebra containing the $\sigma$-algebra generated by $\prc^0(\rt)$,
then for any random variable $U$ defined on the original probability space we have
\[ \Est{U \mid \F} = \E{U \mid \F}\,. \]
This is to be interpreted by restricting $\F$ to $\{\prc^0(\rt) \equiv 1\}$ on the left and the random variable
$\E{U \mid \F}$ to $\{\prc^0(\rt) \equiv 1\}$ on the right.
\end{lem}

\begin{proof}
Suppose that $Z$ is a $\F$-measurable random variable. Then,
\begin{equation*}
\Est{Z\,\E{U \mid \F}} = \frac{\E{\prc^0(\rt) Z\,\E{U \mid \F}}}{\E{\prc^0(\rt)}}\\
= \frac{\E{\E{\prc^0(\rt) Z \,U \mid \F}}}{\E{\prc^0(\rt)}}.\end{equation*}
The last equality is because $\prc^0(\rt)$ and $Z$ are $\F$-measurable. From the
definition of conditional expectation,
\begin{equation*}
\frac{\E{\E{\prc^0(\rt) Z \,U \mid \F}}}{\E{\prc^0(\rt)}} = 
\frac{\E{\prc^0(\rt) Z U}}{\E{\prc^0(\rt)}}\\
= \Est{ZU}. \end{equation*}
The lemma follows from the definition of conditional expectation.
\end{proof}

Define a $[0,1]$-valued random variable $Q_{d,p} = Q_d(\ber^p, \iid^0_d)$,
which we denote the \emph{stability}, on the restricted probability space as follows.
Set
\begin{equation} \label{eqn:stability}
Q_{d,p} = \Est{\prc^1(\rt) \mid \iid^0_d, \ber^p} = \E{\prc^1(\rt) \mid \iid^0_d, \ber^p}\,.\end{equation}
The second equality follows from Lemma \ref{lem:restriction}. In an intuitive sense,
the stability is the conditional probability, given the root is included in the percolation process,
that it remains to be included after re-randomizing the labels.
We now observe that the ratio of the intersection densities are moments of the stability.

\begin{lem} \label{lem:stability}
For every $i \geq 1$ we have that
$$\Est{Q_{d,p}^{i-1}} = \frac{\alpha_{i,d}}{\alpha_{1,d}}.$$
\end{lem}

\begin{proof}
Note that $Q_{d,p} = \E{\prc^j(\rt) \mid \iid^0_d, \ber^p}$ for every $j$ from (\ref{eqn:stability})
and symmtery. Hence,
\begin{eqnarray*} \Est{Q_{d,p}^{i-1}} &=& \frac{\E{\prc^0(\rt) \, (\E{\prc^1(\rt) \mid \iid^0_d, \ber^p})^{i-1}}}{\E{\prc^0(\rt)}}\\
&=& \frac{\E{\prc^0(\rt) \, \left(\prod_{j=1}^{i-1}\E{\prc^j(\rt) \mid \iid^0_d, \ber^p}\right)}}{\E{\prc^0(\rt)}}\,.
\end{eqnarray*}

Observe that the $\prc^j$ are independent of each other conditioned on $(\iid^0_d, \ber^p)$. Hence,
\[ \prod_{j=1}^{i-1}\E{\prc^j(\rt) \mid \iid^0_d, \ber^p} = \E{\prod_{j=1}^{i-1} \prc^j(\rt) \mid \iid^0_d, \ber^p}\,.\]
Since $\prc^0$ is measurable with respect to $\iid^0_d$ we see that
\begin{eqnarray*}
\E{\prc^0(\rt) \, \left(\prod_{j=1}^{i-1}\E{\prc^j(\rt) \mid \iid^0_d, \ber^p}\right)} &=& \E{\prc^0(\rt) \prod_{j=1}^{i-1} \prc^j(\rt)} \\
&=& \pr{\prc^j(\rt) = 1,\; 0 \leq j \leq i-1} \,.
\end{eqnarray*}
Thus, $\Est{Q_{d,p}^{i-1}} = \frac{\pr{\prc^j(\rt) = 1, \; 0 \leq j \leq i-1}}{\den{\prc_d}} = \frac{\alpha_{i,d}}{\alpha_{1,d}}$.
\end{proof}

We will require the following continuity lemma about the stability.
\begin{lem} \label{lem:continuity}
For each $u \geq 0$, the moment $\Est{Q_{d,p}^u}$ is a continuous function of $p$.
When $p=0$, $\Est{Q_{d,p}^u} = 1$, and when $p=1$, $\Est{Q_{d,p}^u} = \den{\prc_d}^u$.
\end{lem}

\begin{proof}
Recall we had assumed that the factor $f_{\prc_d}$ is a block factor with some radius $r = r_d < \infty$.
The parameter $p$ enters into $\Est{Q_{d,p}^u}$ through the random finite set $\{v \in V(\T_{d,r}): \ber^p(v) = 1\}$.
If $S \subset V(\T_{d,r})$ then 
$$ \pr{\ber^p(v) = 1 \;\text{for}\; v \in S \; \text{and}\; \ber^p(v) = 0 \;\text{for}\; v \in V(T_{d,r}) \setminus S}
= p^{\# S}(1-p)^{\# (V(T_{d,r}) \setminus S)}\,.$$
This is a polynomial in $p$, and by conditioning on the output of $\ber^p$ restricted to $V(\T_{d,r})$,
it follows that $\Est{Q_{d,p}^u}$ can be expressed as a convex combination of terms that are free of $p$ with
coefficients given by these probabilities. Thus, $\Est{Q_{d,p}^u}$ is a polynomial in $p$ as well.

When $p=0$ the process $\ber^p \equiv 0$, and $\prc^0 = \prc^1$. Therefore, conditioning on $\iid^0_d$ and restricting
to $\{\prc^0_d(\rt) \equiv 1\}$ makes $Q_{d,0} \equiv 1$. When $p=1$ the process $\ber^p \equiv 1$. So $\prc^1$
becomes independent of $\iid^0_d$. Then the conditioning has no effect and $Q_{d,1} = \E{\prc^1(\rt)} = \den{\prc_d}$.
\end{proof}

\subsection{Completing the proof} \label{sec:mainproof}
Let $\{\prc_d\}$ be a sequence of FIID percolation processes with $\prc_d \in \{0,1\}^{V(\T_d)}$.
For the sake of a contradiction we assume that $\corr{\prc_d} \to 0$ while 
$\den{\prc_d} \geq \alpha \frac{\log d}{d}$ for some $\alpha > 1$. (Technically speaking, we need
to move to a subsequence in $d$, but we can assume WLOG that this is the case.)
Note that it follows from Theorem \ref{thm:corr} that $\alpha \leq 2$. Because of this
we may also assume WLOG that $\den{\prc_d} \leq 10 \frac{\log d}{d}$ for every $d$.

For fixed values of $k \geq 1$ and $0 \leq p \leq 1$ consider the coupling $(\prc^1_d, \ldots, \prc^k_d)$
described in Section \ref{sec:coupling}. Denote the edge profile of this coupled process as $(P_d, \pi_d)$ where
$P_d = [P_d(S,T)]_{S, T \subset \{1,\ldots,k\}}$. The arguments in this section will involve asymptotics
for large $d$ with the values of $k$ and $p$ fixed. As we assumed that $\den{\prc_d} \leq 10 \frac{\log d}{d}$,
for any subset $S \neq \emptyset$, the quantity
\begin{equation} \label{eqn:pibound} \pi_d(S) \leq \pi(\{1\}) \leq \den{\prc_d} \leq 10 \frac{\log d}{d}.\end{equation}
Let us write $\corr{\prc_d} = \eps_d$ so that $\eps_d \to 0$ as $d \to \infty$ by assumption.
Then we have from (\ref{eqn:pstbound}) that
\begin{equation} \label{eqn:pstbound2}
P_d(S,T) \leq \eps_d \frac{\log^2 d}{d^2}\;\text{for every pair}\; S,T\;\text{such that}\; S\cap T \neq \emptyset.\end{equation}

We apply the upper bound from Lemma \ref{lem:maxentropy} to the entropy functional
associated to the edge profile $(P_d, \pi_d)$. We take $\chi = \{0,1\}^k$, the role of the
element $0$ is taken by the empty set, and $\Lambda = \{ (S, T): S \cap T \neq \emptyset \}$.
We may take $K = 100 \eps_d \frac{\log^2 d}{d^2}$, which follows from (\ref{eqn:pstbound2}).
Also, we may take $J = 10 \frac{\log d}{d}$.
Applying Lemma \ref{lem:maxentropy} we conclude that for some $d_0 = d_0(k)$,
if $d \geq d_0$ then
\begin{equation} \label{eqn:coupledentropy}
\frac{d}{2} H(P_d) - (d-1)H(\pi_d) \leq H(\pi_d) - \frac{d}{2}\pi_d^2(\Lm) +
4^k \,[\eps_d + h(\eps_d)] \,\frac{\log^2 d}{d}.\end{equation}

We bound the r.h.s.~of (\ref{eqn:coupledentropy}).
For $S \neq \emptyset$, define $\beta(S)$ by writing $\pi_d(S) = \beta(S) \frac{\log d}{d}$.
We ignore explicitly writing the dependence of $\beta(S)$ on $d$.
With this notation we have
$$\pi_d^2(\Lm)  = \sum_{(S,T): S \cap T \neq \emptyset} \beta(S) \beta(T) \frac{\log^2 d}{d^2}.$$

\begin{lem} \label{lem:coupledentropy}
The entropy functional for $(P_d,\pi_d)$ satisfies the following for $d \geq d_0$:
\begin{align*}
\frac{d}{2} H(P_d) - (d-1)H(\pi_d) & \leq 
\left (\sum_{S \neq \emptyset} \beta(S) - \frac{1}{2} \sum_{(S,T): S \cap T \neq \emptyset} \beta(S) \beta(T) \right ) \frac{\log^2 d}{d} +\\
& 10 \cdot 2^k \cdot \frac{\log d}{d} + 4^k \,[\eps_d + h(\eps_d)] \,  \frac{\log^2 d}{d}.
\end{align*}
\end{lem}

\begin{proof}
We consider the r.h.s.~of (\ref{eqn:coupledentropy}).
For $S \neq \emptyset$, from the identity (\ref{eqn:h1}) we have that
$h(\pi_d(S)) \leq \beta(S) \frac{\log^2 d}{d} + h(\beta(S)) \frac{\log d}{d}$. Now, 
$h(\beta(S)) \leq 1$ since $h(x) \leq 1$ for all $x \geq 0$. Thus,
$h(\pi_d(S)) \leq \beta(S) \frac{\log^2d}{d} + \frac{\log d}{d}$.
On the other hand,
$$\pi_d(\emptyset) = 1 - \sum_{S \neq \emptyset} \beta(S) \frac{\log d}{d}, \;\;\text{and}\;\;
\sum_{S \neq \emptyset} \beta(S) \frac{\log d}{d} \leq 10 \cdot 2^k \cdot \frac{\log d}{d}
\;\;\text{since}\;\; \beta(S) \leq 10 \;\;\text{by}~(\ref{eqn:pibound}).$$ 
Hence, $h(\pi_d(\emptyset)) \leq 10\cdot 2^k\cdot \frac{\log d}{d}$ because $h(1-x) \leq x$.
Therefore,
$$ H(\pi_d) - \frac{d}{2}\pi_d^2(\Lm) \leq  \left (\sum_{S \neq \emptyset} \beta(S) -
\frac{1}{2} \sum_{(S,T): S \cap T \neq \emptyset} \beta(S) \beta(T) \right ) \frac{\log^2 d}{d} + \frac{11\cdot 2^k\cdot \log d}{d}.$$
The desired conclusion follows from the inequality in (\ref{eqn:coupledentropy}).
\end{proof}

The bound provided in Lemma \ref{lem:coupledentropy} depends on the coupling
parameter $p$, which so far has been held fixed, only through the $\beta(S)$'s. We
may take an infimum over $p \in [0,1]$ in the inequality given in Lemma \ref{lem:coupledentropy}.
Moreover, the term $\eps_d + h(\eps_d) \to 0$ as $d \to \infty$. Hence, the top order term
on the r.h.s.~of the inequality presented in Lemma \ref{lem:coupledentropy} is the term
of order $\frac{\log^2 d}{d}$. Then, due to the non-negativity of the entropy functional
(Theorem \ref{thm:FIIDentropy}) we deduce from Lemma \ref{lem:coupledentropy} that
for every $k \geq 1$,
\begin{equation} \label{eqn:betabound}
 \liminf_{d \to \infty} \, \inf_{p \in [0,1]} \; \sum_{S \neq \emptyset} \beta(S) -
 \frac{1}{2} \sum_{(S,T): S \cap T \neq \emptyset} \beta(S) \beta(T) \geq 0.
\end{equation}

For the analysis of (\ref{eqn:betabound}) it is convenient to parametrize $\beta(S)$ through
the intersection densities of the coupled process, as defined in (\ref{eqn:alpha}),
and then express (\ref{eqn:betabound}) in terms of those densities.
The principle of inclusion and exclusion provides the following relation between the $\beta(S)$ and the 
intersection densities $\alpha(S)$ defined in (\ref{eqn:alpha}).
\begin{eqnarray}
\label{eqn:pieatob}
\alpha(S) &=& \sum_{T: S \subset T} \beta(T) \\
\label{eqn:piebtoa}
\beta(S) &=& \sum_{T: S \subset T} (-1)^{\# (T \setminus S)} \alpha(T) \,.
\end{eqnarray}

We now show that the quantity in (\ref{eqn:betabound}) can be rewritten as follows.
\begin{lem} \label{lem:betaalphaidn}
\begin{equation} \label{betaalphaidn}
\sum_{S \neq \emptyset} \beta(S) - \frac{1}{2} \sum_{(S,T): S \cap T \neq \emptyset} \beta(S) \beta(T) =
\sum_{i} (-1)^{i-1}\binom{k}{i} (\alpha_{i,d,p} - \frac{1}{2}\alpha^2_{i,d,p})\,.
\end{equation}
\end{lem}

\begin{proof}
The term $\sum_{S \neq \emptyset} \beta(S)$ equals $\sum_i (-1)^{i-1}\binom{k}{i} \alpha_{i,d,p}$ because both these terms
are equal to the normalized density $\pr{ \prc^i_d(\rt) = 1\; \text{for some}\; i \leq k} \cdot \frac{d}{\log d}$. The relations (\ref{eqn:pieatob})
and (\ref{eqn:piebtoa}) imply that the quantity $\alpha(S)^2 = \sum_{(T_1, T_2): S \subset T_1 \cap T_2} \beta(T_1) \beta(T_2)$.
Thus,

\begin{eqnarray*}
\sum_{i=1}^k (-1)^{i-1} \binom{k}{i} \alpha_{i,d,p}^2 &=& - \sum_{S \neq \emptyset} (-1)^{\# S}\alpha(S)^2 \\
&=& - \sum_{S \neq \emptyset} (-1)^{\# S} \sum_{(T_1, T_2): S \subset T_1 \cap T_2} \beta(T_1) \beta(T_2) \\
&=& - \sum_{(T_1,T_2)} \beta(T_1) \beta(T_2) \sum_{S: S \subset T_1 \cap T_2, S \neq \emptyset} (-1)^{\# S}\\
&=& - \sum_{(T_1,T_2): T_1 \cap T_2 \neq \emptyset} \beta(T_1) \beta(T_2) \sum_{i=1}^{\#\, T_1 \cap T_2} (-1)^{i}\binom{\# \,T_1 \cap T_2}{i}\,.
\end{eqnarray*}

Recall the binomial identity $\sum_{i=1}^t (-1)^{i} \binom{t}{i} = (1-1)^t - 1 = -1$ for any integer $t \geq 1$.
Using this to simplify the last term from the equations above gives
$$\sum_{i=1}^k (-1)^{i-1} \binom{k}{i} \alpha_{i,d,p}^2 = \sum_{(S,T): S \cap T  \neq \emptyset} \beta(S) \beta(T)\,, \; \text{as required}.$$
\end{proof}

We now express (\ref{eqn:betabound}) in terms of the stability via Lemma \ref{lem:betaalphaidn}.
Let $Q_{d,p}$ denote the stability of $\prc_d$ as defined in (\ref{eqn:stability}).
Recall that $\alpha_{i,d,p} = \alpha_{1,d} \Est{Q_{d,p}^{i-1}}$ for all $i \geq 1$ from Lemma \ref{lem:stability}.
Henceforth, we denote the operator $\Est{\cdot}$ by $\E{\cdot}$. Let $R_{d,p}$ denote an
independent copy of $Q_{d,p}$. Then $\alpha_{i,d}^2 = \alpha_{1,d}^2 \E{(Q_{d,p}R_{d,p})^{i-1}}$.
Using the identity
$\sum_{i=1}^{k} (-1)^{i-1} \binom{k}{i} x^{i-1} =  \frac{1 - (1-x)^k}{x}$,
we translate the inequality from (\ref{eqn:betabound}) via the identity (\ref{betaalphaidn}) into

\begin{equation} \label{eqn:Qinq}
\liminf_{d \to \infty} \; \inf_{p \in [0,1]} \; \alpha_{1,d} \E{ \frac{1 - (1-Q_{d,p})^k}{Q_{d,p}}} -
\frac{\alpha_{1,d}^2}{2} \E{ \frac{1 - (1-Q_{d,p}R_{d,p})^k}{Q_{d,p}R_{d,p}}} \geq 0\,.
\end{equation}

Now we make a choice for the value of the coupling parameter $p$ for each value
of $d$.  Fix a parameter $u > 0$ and for all large $d$, pick $p = p(d,u)$ in the
construction of the coupling so that $$\E{Q_{d,p}^u} = 1/ \alpha.$$
This can be done due to the continuity of $p \to \E{Q_{d,p}^u}$ as given by Lemma \ref{lem:continuity}
and the assumption that $\alpha > 1$. We denote $Q_{d,p(d,u)}$ by $Q_d$
for convenience. As of now it is not at all clear as to why we set $p$ in this manner.
However, in the following argument we will see that setting $p$ this way is a judicious choice.

We now explain how to take a limit of $Q_d$ as $d \to \infty$ so that we may analyze (\ref{eqn:Qinq}) in terms of $k$.
Since probability distributions on $[0,1]$ are compact in the weak topology by Prokhorov's Theorem,
we choose a subsequence $(Q_{d_i}, R_{d_i})$ such that it converges in the weak limit to $(Q, R)$.
The random variables $Q$ and $R$ are independent and identically distributed with values in $[0,1]$.

Set $$s_k(x) = \frac{1 - (1-x)^k}{x} \;\;\text{for}\;\; 0 \leq x \leq 1.$$
Observe that $s_k(x) = 1 + (1-x) + \cdots + (1-x)^{k-1}$.
Thus, $s_k(x)$ is continuous, decreasing on $[0,1]$ with maximum value $s_k(0) = k$ and minimum value $s_k(1) = 1$.
Continuity of $s_k$ implies that $\E{s_k(Q_{d_i})} \to \E{s_k(Q)}$ and $\E{s_k(Q_{d_i}R_{d_i})} \to \E{s_k(QR)}$
by the definition of weak convergence of random variables.
We now take limits along the subsequence $d_i$ in (\ref{eqn:Qinq}). The quantity $\alpha_{1,d}$ may not
have a limit but since we assume, for sake of a contradiction, that $\alpha_{1,d} \geq \alpha > 1$, we deduce from taking
limits in the $d_i$ that for all $k \geq 1$
\begin{equation} \label{Qinq2} \E{s_k(Q)} \geq \frac{\alpha}{2} \E{s_k(QR)} .\end{equation}

\subsection{Final step: large $k$ analysis} \label{sec:finalstep}

We want to take the limit in $k$ of (\ref{Qinq2}) as well but we must be careful. The function $s_k(x)$
monotonically converges to $1/x$ for $x \in [0,1]$. So by taking limits in $k$ in (\ref{Qinq2}) and
using the independence of $Q$ and $R$ we conclude that $2 \E{1/Q} \geq \alpha \E{1/Q}^2$.
Of course, we do not know a priori that $\E{1/Q} < \infty$. Even if it were, we can only conclude
that $\alpha \leq 2/ \E{1/Q}$, which leaves us with the seemingly contradictory task of showing that
$\E{1/Q}$ is large but finite.

To deal with these issues we have to use the fact that we have chosen $p = p(d,u)$
such that $\E{Q_d^u} = 1/ \alpha$. In particular, as $x \to x^u$ is continuous for $0 \leq x \leq 1$,
we see that $\E{Q^u} = \lim_{i \to \infty} \E{Q_{d_i}^u} = 1/\alpha$. We have to control the
expectation $\E{1/Q}$ through our control of $\E{Q^u}$.

Three cases can arise: $\pr{Q = 0} > 0$, or $\pr{Q=0} = 0$ but $\E{1/Q} = \infty$, or $\E{1/Q} < \infty$.

\paragraph{\textbf{Case 1:}} $\pr{Q = 0} = q > 0$. The majority of the contribution to $\E{s_k(Q)}$ results from $\{Q=0\}$. More precisely,
$\frac{s_k(x)}{k} \to \mathbf{1}_{x=0}$ as $k \to \infty$, and $\frac{s_k(x)}{k} \in [0,1]$ for all $k$ and $x \in [0,1]$.
From the bounded convergence theorem we deduce that $\E{s_k(Q)/k} \to \pr{Q=0}$ and $\E{s_k(QR)/k} \to \pr{QR=0}$
as $k \to \infty$. The latter probability is $2q - q^2$ due to $Q$ and $R$ being independent and identically distributed.
By dividing the inequality in (\ref{Qinq2}) by $k$ and taking a limit we conclude that
\[ 2q - \alpha(2q-q^2) \geq 0, \; \text{or equivalently}, \; \alpha \leq \frac{2}{2-q}\,.\]

Now, since for $x \in [0,1]$, we have that $\mathbf{1}_{x=0} \leq 1 - x^u$. It follows from here that $q \leq 1 - \E{Q^u} = 1 - 1/\alpha$. Hence,
$$ \alpha \leq \frac{2}{2-q} \leq \frac{2}{1 + \alpha^{-1}}\,.$$
Simplifying the latter inequality gives $ \alpha \leq 1$; a contradiction.

\paragraph{\textbf{Case 2:}} $\pr{Q=0} = 0$ but $\E{\frac{1}{Q}} = \infty$. Now, most of the contribution to $\E{s_k(Q)}$ occurs when $Q$ is small.
Recall that $s_k(x) \nearrow 1/x$ as $k \to \infty$. By the monotone convergence theorem, $\E{s_k(Q)} \to \infty$ as $k \to \infty$.

Fix $0 < \epsilon < 1$, and write $s_k(x) = s_{k, \leq \eps}(x) + s_{k, > \eps}(x)$ where $s_{k, \leq \eps}(x) = s_k(x)\,\mathbf{1}_{x \leq \eps}$.
Note that $s_{k, > \eps}(x) \leq \eps^{-1}$ for all $k$. We have that
\begin{equation} \label{eqn:supbound} \E{s_k(Q)} = \E{s_{k, \leq \eps}(Q)} +  \E{s_{k, > \eps}(Q)} \leq \E{s_{k, \leq \eps}(Q)} + \eps^{-1}.\end{equation}
Therefore, $\E{s_{k, \leq \eps}(Q)} \to \infty$ with $k$ since $\E{s_k(Q)} \to \infty$.

Observe from the positivity of $s_k$ that
\[ \E{s_k(QR)} \geq \E{s_k(QR); Q \leq \eps, R > \eps} + \E{s_k(QR); Q > \eps, R \leq \eps}.\]
The latter two terms are equal by symmetry, so $\E{s_k(QR)} \geq 2 \E{s_k(QR); Q \leq \eps, R > \eps}$.
The fact that $s_k(x)$ is decreasing in $x$ and $R \leq 1$ imply that $s_k(QR) \geq s_k(Q)$.
Together with the independence of $Q$ and $R$ we deduce that
\[ \E{s_k(QR); Q \leq \eps, R > \eps} \geq \E{s_k(Q); Q \leq \eps, R > \eps} = \E{s_{k, \leq \eps}(Q)} \pr{R > \eps}.\]
Consequently,
\begin{equation} \label{eqn:slwbound} \E{s_k(QR)}  \geq 2 \E{s_{k, \leq \eps}(Q)} \pr{Q > \eps}.\end{equation}

The inequality in (\ref{Qinq2}) is $\frac{\alpha}{2} \leq \frac{\E{s_k(Q)}}{\E{s_k(QR)}}$. The bounds from (\ref{eqn:supbound}) and (\ref{eqn:slwbound})
imply that
$$\frac{\alpha}{2} \leq \frac{\E{s_{k, \leq \eps}(Q)} + \eps^{-1}}{2 \E{s_{k, \leq \eps}(Q)} \pr{Q > \eps}}\,.$$
Since $\E{s_{k, \leq \eps}(Q)} \to \infty$ with $k$ we can take a limit in $k$ to conclude that
$$ \alpha \leq \frac{1}{\pr{Q > \eps}}.$$
As $\eps \to 0$ the probability $\pr{Q > \eps} \to \pr{Q > 0} = 1$, by assumption. Thus, $\alpha \leq 1$, a contradiction.

\paragraph{\textbf{Case 3:}} $\E{1/Q}$ is finite. Since $s_k(x)$ monotonically converges to $1/x$ as $k \to \infty$, we deduce
from (\ref{Qinq2}) that
\[ 2 \E{\frac{1}{Q}} - \alpha \E{\frac{1}{QR}} \geq 0\,.\]
Since $\E{1/(QR)} = \E{1/Q}^2$, the inequality above becomes $\alpha \leq 2 \E{1/Q}^{-1}$.
Jensen's inequality implies that $\E{1/Q}^{-1} \leq \E{Q^u}^{1/u}$ .
However, $\E{Q^u} = 1/\alpha$ and this means that $\alpha \leq 2^{\frac{u}{u+1}}$.
As the previous two cases led to contradictions we conclude that $\alpha \leq 2^{\frac{u}{u+1}}$ for all $u > 0$. However,
as $u \to 0$ we see that $\alpha \leq 1$. This final contradiction completes the proof of Theorem \ref{thm:main}.

\vskip 0.1in
\paragraph{\textbf{Proof of Corollary \ref{cor:fincomp}}}
If $\prc$ is an FIID percolation process on $\T_d$ that has finite components with probability one,
then the components of $\prc$ are all finite trees. Thus, the component of the root provides
a measure on finite, rooted trees with the root being picked uniformly at random (due to invariance
and transitivity). If a tree has $n$ vertices then the expected degree of a uniform random root
is $2(n-1)/n \leq 2$. This implies that the average degree of the root of $\T_d$, given that it is
included in $\prc$, cannot be larger than 2. As such, the hypothesis of Theorem \ref{thm:main}
applies in this situation.

\subsection{Density bounds for small values of $d$} \label{sec:smalld}

We conclude this section with remarks on FIID percolation processes on $\T_d$ for `small'
values of $d$. Although the entropy inequality may be used in principal for every $d$ to derive bounds on the density
of FIID percolation on $\T_d$ with finite clusters, these upper bounds are not the best available
for small values of $d$. The best known explicit upper bounds for small values of $d$ ($d \leq 10$) is due to
Bau et.~al.~\cite{BWZ}. They are derived by using more exhaustive counting arguments on random $d$-regular
graphs than those which provide our entropy inequality. The cases for $d=3$ or $d=4$ are rather interesting.
The upper bound of $2/3$ for $\T_4$ have not be ruled out by further counting arguments, but the best known
lower bound is $0.6045$ due to a construction of Hoppen and Wormald \cite{HWlwbdd}.
We suspect that the upper bound of $2/3$ is optimal.

For $d=3$, Bau et.~al. provide a upper bound of $3/4$.
Endre Cs\'{o}ka has confirmed the upper bound of $3/4$ as optimal in personal communication .
Cs\'{o}ka's construction may be derived from an appropriate edge orientations of $\T_3$.
There exists, as a weak limit of FIID processes, an invariant orientation
of the edges of $\T_3$ such that the out-degree of every vertex is either 1 or 3.
The percolation process consists of the vertices with out-degree 1.
Invariance implies that their density is $3/4$. The clusters are infinite but
have at most two topological ends (which means that no cluster contains three
disjoint infinite rays emanating from a common vertex). The clusters can then
be made finite by performing a Bernoulli percolation on the vertices of out-degree 1
at any density less than 1. The density of the resulting percolation
process is arbitrarily close to $3/4$.

Although Cs\'{o}ka's construction realizes the optimal percolation process
as a weak limit of FIID processes, it is not necessarily true that
the limiting process is itself an FIID process. There exist invariant, but not FIID,
processes on $\T_d$ that are weak limits of FIID process \cite[Theorem 6]{HV}.

\section{FIID edge orientations of $\T_d$} \label{sec:orientations}

Following Cs\'{o}ka's construction of percolation processes via edge orientations,
we may ask what other types of FIID edge orientations are possible on $\T_d$.
In this section we will show that $\T_d$ admits an FIID edge orientation with no sources
or sinks.

FIID processes on the edges on $\T_d$ are defined analogous to vertex-indexed processes.
Since $\mathrm{Aut}(\T_d)$ acts transitively on the edges of $\T_d$ the definition for
vertex-indexed processes carries over naturally. An edge indexed FIID process
$\Phi \in \chi^{E(\T_d)}$ is thus determined by a measurable function $F: [0,1]^{E(\T_d)} \to \chi^{E(\T_d)}$
satisfying $F(\gamma \cdot \omega) = \gamma \cdot F(\omega)$ for all $\gamma \in \mathrm{Aut}(\T_d)$.
$\Phi$ is then defined as $\Phi = F(\iid)$, where $\iid$ is a random labelling of the edges of $\T_d$.

Given any edge orientation of a graph $G$, a vertex $v$ is called a source if all edges that are incident
to $v$ are oriented away from it. If all incident edges are oriented towards $v$, then $v$ is called a sink.
For $d \geq 3$, we use the existence of FIID perfect matchings \cite{LN} to construct an FIID edge orientation
of $\T_d$ with no sources or sinks. Such an orientation cannot exists for $\T_2$. Indeed, $\T_2$ has
only two edge orientations with no sources or sinks: either all edges point to the right or all point to the left.
However, the uniform distribution on these two orientations cannot be an FIID process because it is not mixing.

\begin{thm} \label{thm:orientation}
For every $d \geq 3$ there exists an FIID orientation of the edges of $\T_d$ that contains no sources or sinks.
\end{thm}

\begin{figure}[!ht]
\begin{center}
\includegraphics[scale=.75]{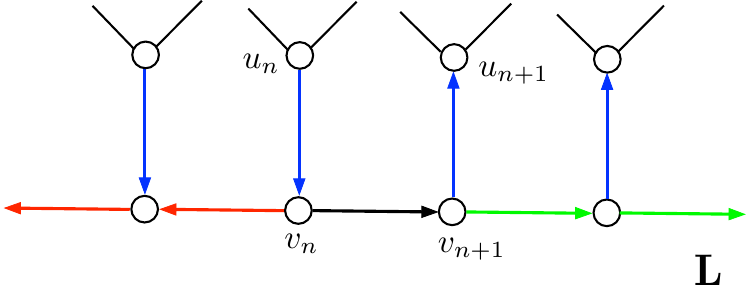}
\end{center}
\caption{Orientation of a bi-infinite path $\mathbf{L}$. Edges of the
perfect matching that meet $\mathbf{L}$ are in blue. The edges of $L_n$ are red
and edges of $L_{n+1}$ are green. The black edge between $v_n$ and $v_{n+1}$
connects $L_n$ to $L_{n+1}$.}\label{fig:orientation}	
\end{figure}

\begin{proof}
Lyons and Nazarov prove that $\T_d$ has an FIID perfect matching \cite{LN}.
We consider an FIID perfect matching of $\T_d$ and orient its edges independently and uniformly at random.
Then we remove the oriented edges of the matching and are left with a forest whose components are isomorphic to $\T_{d-1}$.
We continue with this procedure of selecting an FIID perfect matching from the components, orienting those edges
at random and then removing them from the graph, until we are left with a forest whose components are isomorphic to $\T_3$.
Thus, we have to orient the edges of $\T_3$ with no sources or sinks as an FIID process to finish.

Henceforth, we explain how to orient the edges of $\T_3$ in such a manner. Begin by finding an FIID perfect matching
on $\T_3$ and then orient the edges of the matching independently and uniformly at random. The un-oriented
edges span a 2-factor, that is, its components are un-oriented bi-infinite paths. Consider any such path $\mathbf{L}$.
The oriented perfect matching partitions $\mathbf{L}$ into contiguous finite paths $L_{n}, n \in \mathbb{Z}$,
which are characterized by the following properties.

First, $L_n$ is incident to $L_{n+1}$ via an edge $\{v_n,v_{n+1}\}$ in $\mathbf{L}$.
Second, for every $n$, the edges of the perfect matching that meet $L_n$ are all oriented
in `parallel', either all pointing towards $L_n$ or away from it. Finally,
the edges of the matching that meet $L_{n+1}$ are oriented in the opposite direction
compared to the (common) orientation of those edges of the matching that meet $L_n$.
This setup is shown in Figure \ref{fig:orientation}.
The reason the paths $L_n$ are finite is because the edges of the matching that are
incident to $\mathbf{L}$ are oriented independently. We can think of the $L_n$s that are pointing
towards the matching as the clusters of a Bernoulli percolation on $\mathbf{L}$ at density $1/2$.

To complete the orientation we first orient all the edges on the path $L_n$ in the same direction.
As there are two possible directions, we choose one at random. This is
done independently for each $L_n$ on every un-oriented bi-infinite path $\mathbf{L}$.
The finiteness of the $L_n$ is crucial to ensure that these orientations can be done as an FIID.
Following this, any vertex in the interior of the path $L_n$ has one
edge directed towards it and another directed away. This ensures that such a
vertex cannot form a source or a sink.

The vertices that can form sources or sinks are those at the endpoints of any path $L_n$.
Also, the un-oriented edges that remain are the $\{v_n, v_{n+1}\}$ which connect two contiguous
paths $L_n$ and $L_{n+1}$ on a common bi-infinite path $\mathbf{L}$. Observe that the endpoints of these
edges are precisely the vertices that can form sources and sinks (see the black edge in Figure \ref{fig:orientation}).
Given such an edge $\{v_n, v_{n+1}\}$ let $u_n$ (resp. $u_{n+1}$) be the neighbour of $v_n$ (resp. $v_{n+1})$
such that the edge $\{v_n,u_n\}$ (resp. $\{v_{n+1},u_{n+1}\}$) lies in the matching.
The edges $\{v_n,u_n\}$ and $\{v_{n+1},u_{n+1}\}$ have opposite orientations by design (see Figure \ref{fig:orientation}).
Suppose that $u_n \to v_n$ and $v_{n+1} \to u_{n+1}$.
We orient $\{v_n,v_{n+1}\}$ as $v_n \to v_{n+1}$, thus ensuring that both $v_n$ and $v_{n+1}$
have one incoming edge and another outgoing one. This completes the orientation of
$\T_3$ and produces no sources or sinks.

\end{proof}

\subsection*{Acknowledgements}
The Authors research has been supported by an NSERC CGS grant.
The author is grateful to Endre Cs\'{o}ka and Viktor Harangi for explaining the
construction of the optimal percolation process of density $3/4$ in Section \ref{sec:smalld}.
The author also thanks B\'{a}lint Vir\'{a}g for helpful discussions and the anonymous referees
for helpful suggestions regarding the exposition.

\end{document}